\documentclass[arXiv]{Feanor}

\usepackage{amsmath}
\usepackage{amsthm}
\usepackage{amssymb}
\usepackage{bbm}
\usepackage{bm}
\usepackage{stmaryrd}
\usepackage{tikz}
\usepackage{mathtools}
\counterwithin{equation}{section}
\usepackage{mathrsfs}

\usepackage{pgfplots}
\pgfplotsset{compat=1.18}
\usepackage{subcaption}
\usepackage[hidelinks]{hyperref}
\usepackage[stretch=10]{microtype}
\usepackage{float}
\usepackage{graphicx}
\usepackage{caption}
\usepackage{url}
\usepackage{enumitem}
\usepackage{cleveref}
\crefformat{equation}{(#2#1#3)}
\usepackage[normalem]{ulem}
\usepackage{booktabs}
\usepackage[numbers,sort]{natbib}

\usepackage[framemethod=TikZ]{mdframed}
\usepackage{xcolor}
\usepackage{siunitx}

\theoremstyle{plain}
\newtheorem{theorem}{Theorem}[section]

\newtheorem{corollary}[theorem]{Corollary}

\newtheorem{lemma}[theorem]{Lemma}
\newtheorem{conjecture}[theorem]{Conjecture}

\newtheorem{proposition}[theorem]{Proposition}

\theoremstyle{definition}
\newtheorem{definition}[theorem]{Definition}
\newtheorem{remark}[theorem]{Remark}

\newcommand{\eg}{{e.g.,}\ }

\newcommand{\nref}[2]{\hyperref[#1]{#2}} 
\newcommand{\negquad}{\hspace{-1em}}
\newcommand{\negsp}{\hspace{-0.25em}}

\newcommand{\de}{\vcentcolon=}

\newcommand{\Unif}[1]{\operatorname{Unif}#1}

\newcommand{\T}{{\rm T}}
\newcommand{\spec}{\operatorname{spec}}

\newcommand{\spn}{\operatorname{span}}
\newcommand{\coker}{\operatorname{coker}}
\newcommand{\diag}{\operatorname{diag}}
\newcommand{\rank}{\operatorname{rank}}
\newcommand{\Hom}{\operatorname{Hom}}
\newcommand{\Aut}{\operatorname{Aut}}
\newcommand{\Sur}{\operatorname{Sur}}
\newcommand{\Ext}{\operatorname{Ext}}

\newcommand{\defbold} [1]{\expandafter\newcommand\csname b#1\endcsname{\mathbf{#1}}}
\newcommand{\defcal} [1]{\expandafter\newcommand\csname c#1\endcsname{\mathcal{#1}}}
\newcommand{\defbb} [1]{\expandafter\newcommand\csname bb#1\endcsname{\mathbb{#1}}}
\newcommand{\defscr} [1]{\expandafter\newcommand\csname s#1\endcsname{\mathscr{#1}}}
\newcommand{\deffrak} [1]{\expandafter\newcommand\csname f#1\endcsname{\mathfrak{#1}}}
\newcommand{\defbracketed} [1]{\expandafter\newcommand\csname bc#1\endcsname{\bc{#1}}}
\newcommand{\deftilde} [1]{\expandafter\newcommand\csname ti#1\endcsname{\tilde{#1}}}
\newcommand{\defoverline} [1]{\expandafter\newcommand\csname ol#1\endcsname{\overline{#1}}}
\newcommand{\defvec} [1]{\expandafter\newcommand\csname vec#1\endcsname{\vec{#1}}}

\expandafter\defbold{A}
\expandafter\defbold{B}
\expandafter\defbold{C}
\expandafter\defbold{D}
\expandafter\defbold{E}
\expandafter\defbold{F}
\expandafter\defbold{G}
\expandafter\defbold{H}
\expandafter\defbold{I}
\expandafter\defbold{J}
\expandafter\defbold{K}
\expandafter\defbold{L}
\expandafter\defbold{M}
\expandafter\defbold{N}
\expandafter\defbold{O}
\expandafter\defbold{P}
\expandafter\defbold{Q}
\expandafter\defbold{R}
\expandafter\defbold{S}
\expandafter\defbold{T}
\expandafter\defbold{U}
\expandafter\defbold{V}
\expandafter\defbold{W}
\expandafter\defbold{X}
\expandafter\defbold{Y}
\expandafter\defbold{Z}

\expandafter\defcal{A}
\expandafter\defcal{B}
\expandafter\defcal{C}
\expandafter\defcal{D}
\expandafter\defcal{E}
\expandafter\defcal{F}
\expandafter\defcal{G}
\expandafter\defcal{H}
\expandafter\defcal{I}
\expandafter\defcal{J}
\expandafter\defcal{K}
\expandafter\defcal{L}
\expandafter\defcal{M}
\expandafter\defcal{N}
\expandafter\defcal{O}
\expandafter\defcal{P}
\expandafter\defcal{Q}
\expandafter\defcal{R}
\expandafter\defcal{S}
\expandafter\defcal{T}
\expandafter\defcal{U}
\expandafter\defcal{V}
\expandafter\defcal{W}
\expandafter\defcal{X}
\expandafter\defcal{Y}
\expandafter\defcal{Z}

\expandafter\defbb{A}
\expandafter\defbb{B}
\expandafter\defbb{C}
\expandafter\defbb{D}
\expandafter\defbb{E}
\expandafter\defbb{F}
\expandafter\defbb{G}
\expandafter\defbb{H}
\expandafter\defbb{I}
\expandafter\defbb{J}
\expandafter\defbb{K}
\expandafter\defbb{L}
\expandafter\defbb{M}
\expandafter\defbb{N}
\expandafter\defbb{O}
\expandafter\defbb{P}
\expandafter\defbb{Q}
\expandafter\defbb{R}
\expandafter\defbb{S}
\expandafter\defbb{T}
\expandafter\defbb{U}
\expandafter\defbb{V}
\expandafter\defbb{W}
\expandafter\defbb{X}
\expandafter\defbb{Y}
\expandafter\defbb{Z}

\expandafter\defscr{A}
\expandafter\defscr{B}
\expandafter\defscr{C}
\expandafter\defscr{D}
\expandafter\defscr{E}
\expandafter\defscr{F}
\expandafter\defscr{G}
\expandafter\defscr{H}
\expandafter\defscr{I}
\expandafter\defscr{J}
\expandafter\defscr{K}
\expandafter\defscr{L}
\expandafter\defscr{M}
\expandafter\defscr{N}
\expandafter\defscr{O}
\expandafter\defscr{P}
\expandafter\defscr{Q}
\expandafter\defscr{R}
\expandafter\defscr{S}
\expandafter\defscr{T}
\expandafter\defscr{U}
\expandafter\defscr{V}
\expandafter\defscr{W}
\expandafter\defscr{X}
\expandafter\defscr{Y}
\expandafter\defscr{Z}

\expandafter\deffrak{A}
\expandafter\deffrak{B}
\expandafter\deffrak{C}
\expandafter\deffrak{D}
\expandafter\deffrak{E}
\expandafter\deffrak{F}
\expandafter\deffrak{G}
\expandafter\deffrak{H}
\expandafter\deffrak{I}
\expandafter\deffrak{J}
\expandafter\deffrak{K}
\expandafter\deffrak{L}
\expandafter\deffrak{M}
\expandafter\deffrak{N}
\expandafter\deffrak{O}
\expandafter\deffrak{P}
\expandafter\deffrak{Q}
\expandafter\deffrak{R}
\expandafter\deffrak{S}
\expandafter\deffrak{T}
\expandafter\deffrak{U}
\expandafter\deffrak{V}
\expandafter\deffrak{W}
\expandafter\deffrak{X}
\expandafter\deffrak{Y}
\expandafter\deffrak{Z}

\expandafter\defbracketed{A}
\expandafter\defbracketed{B}
\expandafter\defbracketed{C}
\expandafter\defbracketed{D}
\expandafter\defbracketed{E}
\expandafter\defbracketed{F}
\expandafter\defbracketed{G}
\expandafter\defbracketed{H}
\expandafter\defbracketed{I}
\expandafter\defbracketed{J}
\expandafter\defbracketed{K}
\expandafter\defbracketed{L}
\expandafter\defbracketed{M}
\expandafter\defbracketed{N}
\expandafter\defbracketed{O}
\expandafter\defbracketed{P}
\expandafter\defbracketed{Q}
\expandafter\defbracketed{R}
\expandafter\defbracketed{S}
\expandafter\defbracketed{T}
\expandafter\defbracketed{U}
\expandafter\defbracketed{V}
\expandafter\defbracketed{W}
\expandafter\defbracketed{X}
\expandafter\defbracketed{Y}
\expandafter\defbracketed{Z}

\expandafter\deftilde{A}
\expandafter\deftilde{B}
\expandafter\deftilde{C}
\expandafter\deftilde{D}
\expandafter\deftilde{E}
\expandafter\deftilde{F}
\expandafter\deftilde{G}
\expandafter\deftilde{H}
\expandafter\deftilde{I}
\expandafter\deftilde{J}
\expandafter\deftilde{K}
\expandafter\deftilde{L}
\expandafter\deftilde{M}
\expandafter\deftilde{N}
\expandafter\deftilde{O}
\expandafter\deftilde{P}
\expandafter\deftilde{Q}
\expandafter\deftilde{R}
\expandafter\deftilde{S}
\expandafter\deftilde{T}
\expandafter\deftilde{U}
\expandafter\deftilde{V}
\expandafter\deftilde{W}
\expandafter\deftilde{X}
\expandafter\deftilde{Y}
\expandafter\deftilde{Z}

\expandafter\defoverline{A}
\expandafter\defoverline{B}
\expandafter\defoverline{C}
\expandafter\defoverline{D}
\expandafter\defoverline{E}
\expandafter\defoverline{F}
\expandafter\defoverline{G}
\expandafter\defoverline{H}
\expandafter\defoverline{I}
\expandafter\defoverline{J}
\expandafter\defoverline{K}
\expandafter\defoverline{L}
\expandafter\defoverline{M}
\expandafter\defoverline{N}
\expandafter\defoverline{O}
\expandafter\defoverline{P}
\expandafter\defoverline{Q}
\expandafter\defoverline{R}
\expandafter\defoverline{S}
\expandafter\defoverline{T}
\expandafter\defoverline{U}
\expandafter\defoverline{V}
\expandafter\defoverline{W}
\expandafter\defoverline{X}
\expandafter\defoverline{Y}
\expandafter\defoverline{Z}

\expandafter\defvec{A}
\expandafter\defvec{B}
\expandafter\defvec{C}
\expandafter\defvec{D}
\expandafter\defvec{E}
\expandafter\defvec{F}
\expandafter\defvec{G}
\expandafter\defvec{H}
\expandafter\defvec{I}
\expandafter\defvec{J}
\expandafter\defvec{K}
\expandafter\defvec{L}
\expandafter\defvec{M}
\expandafter\defvec{N}
\expandafter\defvec{O}
\expandafter\defvec{P}
\expandafter\defvec{Q}
\expandafter\defvec{R}
\expandafter\defvec{S}
\expandafter\defvec{T}
\expandafter\defvec{U}
\expandafter\defvec{V}
\expandafter\defvec{W}
\expandafter\defvec{X}
\expandafter\defvec{Y}
\expandafter\defvec{Z}

\expandafter\defbracketed{a}
\expandafter\defbracketed{b}
\expandafter\defbracketed{c}
\expandafter\defbracketed{d}
\expandafter\defbracketed{e}
\expandafter\defbracketed{f}
\expandafter\defbracketed{g}
\expandafter\defbracketed{h}
\expandafter\defbracketed{i}
\expandafter\defbracketed{j}
\expandafter\defbracketed{k}
\expandafter\defbracketed{l}
\expandafter\defbracketed{m}
\expandafter\defbracketed{n}
\expandafter\defbracketed{o}
\expandafter\defbracketed{p}
\expandafter\defbracketed{q}
\expandafter\defbracketed{r}
\expandafter\defbracketed{s}
\expandafter\defbracketed{t}
\expandafter\defbracketed{u}
\expandafter\defbracketed{v}
\expandafter\defbracketed{w}
\expandafter\defbracketed{x}
\expandafter\defbracketed{y}
\expandafter\defbracketed{z}

\expandafter\deftilde{a}
\expandafter\deftilde{b}
\expandafter\deftilde{c}
\expandafter\deftilde{d}
\expandafter\deftilde{e}
\expandafter\deftilde{f}
\expandafter\deftilde{g}
\expandafter\deftilde{h}
\expandafter\deftilde{i}
\expandafter\deftilde{j}
\expandafter\deftilde{k}
\expandafter\deftilde{l}
\expandafter\deftilde{m}
\expandafter\deftilde{n}
\expandafter\deftilde{o}
\expandafter\deftilde{p}
\expandafter\deftilde{q}
\expandafter\deftilde{r}
\expandafter\deftilde{s}
\expandafter\deftilde{t}
\expandafter\deftilde{u}
\expandafter\deftilde{v}
\expandafter\deftilde{w}
\expandafter\deftilde{x}
\expandafter\deftilde{y}
\expandafter\deftilde{z}

\expandafter\defoverline{a}
\expandafter\defoverline{b}
\expandafter\defoverline{c}
\expandafter\defoverline{d}
\expandafter\defoverline{e}
\expandafter\defoverline{f}
\expandafter\defoverline{g}
\expandafter\defoverline{h}
\expandafter\defoverline{i}
\expandafter\defoverline{j}
\expandafter\defoverline{k}
\expandafter\defoverline{l}
\expandafter\defoverline{m}
\expandafter\defoverline{n}
\expandafter\defoverline{o}
\expandafter\defoverline{p}
\expandafter\defoverline{q}
\expandafter\defoverline{r}
\expandafter\defoverline{s}
\expandafter\defoverline{t}
\expandafter\defoverline{u}
\expandafter\defoverline{v}
\expandafter\defoverline{w}
\expandafter\defoverline{x}
\expandafter\defoverline{y}
\expandafter\defoverline{z}

\expandafter\defvec{a}
\expandafter\defvec{b}
\expandafter\defvec{c}
\expandafter\defvec{d}
\expandafter\defvec{e}
\expandafter\defvec{f}
\expandafter\defvec{g}
\expandafter\defvec{h}
\expandafter\defvec{i}
\expandafter\defvec{j}
\expandafter\defvec{k}
\expandafter\defvec{l}
\expandafter\defvec{m}
\expandafter\defvec{n}
\expandafter\defvec{o}
\expandafter\defvec{p}
\expandafter\defvec{q}
\expandafter\defvec{r}
\expandafter\defvec{s}
\expandafter\defvec{t}
\expandafter\defvec{u}
\expandafter\defvec{v}
\expandafter\defvec{w}
\expandafter\defvec{x}
\expandafter\defvec{y}
\expandafter\defvec{z}

\def\b1{\mathbf{1}}
\def\bb1{\mathbbm{1}}

\newcommand{\bc}[1]{{\text{\fontsize{5}{5}\selectfont{$($}{\hspace*{-0.3pt}}\fontsize{6}{4}\selectfont$#1$\hspace*{-0.1pt}\fontsize{5}{5}\selectfont{$)$}}}}

\newcommand{\bAcomplement}{\bA^{\mkern-3mu {c}}}

\begin{document}
\begin{frontmatter}
\title{Cokernel statistics for walk matrices \\ 
of directed and weighted random graphs}

\begin{aug}
    \author{\fnms{Alexander}~\snm{Van Werde}},
    \address{Eindhoven University of Technology, Department of Mathematics and Computer Science\\
    \href{mailto:a.van.werde@tue.nl}{a.van.werde@tue.nl}}
\end{aug}

\begin{abstract}
    The \emph{walk matrix} associated to an $n\times n$ integer matrix $\bX$ and an integer vector $b$ is defined by $\bW \de (b,\bX b,\ldots,\bX^{n-1}b)$.
    We study limiting laws for the cokernel of $\bW$ in the scenario where $\bX$ is a random matrix with independent entries and $b$ is deterministic. 
    Our first main result provides a formula for the distribution of the $p^m$-torsion part of the cokernel, as a group, when $\bX$ has independent entries from a specific distribution. 
    The second main result relaxes the distributional assumption and concerns the $\bbZ[x]$-module structure. 

    The motivation for this work arises from an open problem in spectral graph theory which asks to show that random graphs are often determined up to isomorphism by their (generalized) spectrum. 
    Sufficient conditions for generalized spectral determinacy can namely be stated in terms of the cokernel of a walk matrix.  
    Extensions of our results could potentially be used to determine how often those conditions are satisfied. 
    Some remaining challenges for such extensions are outlined in the paper.  
\end{abstract}

\end{frontmatter}

\renewcommand{\thefootnote}{\fnsymbol{footnote}} 
\footnotetext{
    \emph{\href{https://mathscinet.ams.org/mathscinet/msc/msc2020.html}{MSC2020 subject classification:}} 05C50, 15B52, 60B20.}
\renewcommand{\thefootnote}{\arabic{footnote}} 

\section{Introduction}
What information about a graph is encoded in the spectrum of its adjacency matrix? 
A well-known conjecture by van Dam and Haemers \cite{van2003graphs} suggests that the answer to this question is \emph{all information in the typical case} in the sense that almost all graphs are determined up to isomorphism by their spectrum.
Unfortunately, progress towards that conjecture has been fairly limited due to the fact that there are essentially no known general-purpose methods to prove that a graph is determined by its spectrum. 
There are many more proof techniques available to prove that a graph's spectrum does \emph{not} determine it than to show that it does \cite{schwenk1973almost,godsil1982constructing,halbeisen1999generation,abiad2012cospectral,butler2010note}.

In view of this, it is intriguing that a sufficient condition for \emph{generalized} spectral determinacy was discovered in 2006 by Wang and Xu \cite{wang2006sufficient,wang2006excluding}. 
The \emph{generalized spectrum} of a graph $G$ here refers to the ordered pair $(\spec(\bA), \spec(\bAcomplement))$ consisting of the spectra of the adjacency matrix $\bA$ of $G$ and of the adjacency matrix $\bAcomplement$ of its complement graph.
Refinements of the results of \cite{wang2006sufficient,wang2006excluding} have given rise to an active area of research in recent years; see \eg \cite{wang2017simple,wang2013generalized,qiu2023smith,qiu2021oriented,li2021arithmetic,wang2022haemers}. 
For definiteness, let us state a refinement which is particularly insightful and motivates the results of the current paper. 

Define a matrix with integer entries by $\bW \de (\bA^{j-1}e)_{j=1}^n$ where $e = (1,\ldots,1)^{\T}$ is the all-ones vector and $n$ is the number of vertices of $G$.
Let $\coker(\bW) \de \bbZ^n/\bW(\bbZ^n)$ denote the \emph{cokernel} of this matrix.
Then, the following is an equivalent rephrasing of \cite[Theorem 1.1]{wang2017simple}; see \cite[p.2]{qiu2023smith}. 
\begin{theorem}[Wang, \texorpdfstring{\cite{wang2017simple}}{[]}]\label{thm: Wang}
    Let $G$ be a simple graph.
    Assume that there exists an odd and square-free integer $m$ such that 
    \begin{align} 
        \coker(\bW) \cong (\bbZ/2\bbZ)^{\lfloor n/2 \rfloor} \oplus (\bbZ/m\bbZ ) \nonumber
    \end{align} 
    as an Abelian group. 
    Then, $G$ is determined by its generalized spectrum up to isomorphism.   
\end{theorem}
\newpage
It is believed that the conditions of \Cref{thm: Wang} are satisfied for a nonvanishing fraction of all simple graphs \cite[Conjecture 2]{wang2017simple}.   
For comparison, the best known bound on the non-generalized problem is due to Koval and Kwan \cite{koval2023exponentially} who recently established that there are at least $e^{cn}$ graphs on $n$ vertices which are determined by their spectrum. 
The result from \cite{koval2023exponentially} represents a significant improvement relative to the previous long-standing barrier of $e^{c\sqrt{n}}$ but still only yields a quickly vanishing fraction of all $(1-o(1))2^{n(n-1)/2}/n!$ simple graphs.
So, a nonvanishing fraction being determined by (generalized) spectrum would signify a remarkable development.

However, even though criteria for generalized spectral determinacy have been known for almost 20 years now, it remains an open problem to prove that the criteria are indeed frequently satisfied. 
The current state of knowledge on the frequency of satisfaction is mostly limited to numerical studies \cite{wang2017simple,van2003graphs,haemers2004enumeration,wang2022haemers}.
The lack of theoretical work is surprising given that criteria for generalized spectral determinacy are an active area of research. 
A possible explanation is that it is not clear what proof techniques could be used.  
The current paper develops a novel line of attack by making a connection to proof techniques \cite{cheong2023cokernel,wood2019random,nguyen2020surjectivity,sawin2022moment} which were historically developed in the context of Cohen--Lenstra heuristics for the class groups of number fields \cite{cohen2006heuristics} and in the study of sandpile groups of random graphs \cite{clancy2015cohen}.

The problem is too challenging to solve in a single step, so we direct our efforts to a variant which is more convenient for technical reasons. 
Instead of simple graphs, we study random directed graphs with random edge weights. 
Concretely, given a $\bbZ^{n\times n}$-valued random matrix $\bX$ with independent entries and a vector $b\in \bbZ^n$, we study the cokernel of the associated \emph{walk matrix} $\bW \de (\bX^{j-1} b)_{j=1}^{n}$.
To explain the terminology, note that if $\bX$ is $\{0,1 \}^{n\times n}$-valued and interpreted as the directed adjacency matrix of a directed graph and $b = \bb1_S$ is the indicator vector of a subset $S\subseteq \{1,\ldots,n \}$, then $\bW_{i,j}$ counts the number of walks of length $j-1$ starting from $i$ with an endpoint in $S$.
Walk matrices of this kind are also of independent interest due to applications to control theory \cite{godsil2012controllable,sundaram2012structural,o2016conjecture} and graph isomorphism problems \cite{liu2022unlocking,godsil2012controllable,verbitsky2023canonization}.   

In future work, it would be interesting to pursue extensions of our results to the setting of simple graphs where \Cref{thm: Wang} is applicable\footnote{There is no known analogue of \Cref{thm: Wang} for directed graphs. 
(See however \cite{qiu2021oriented}.) 
Indeed, directed graphs are typically \emph{not} determined by their (generalized) spectrum because the transpose of a directed adjacency matrix has the same spectrum but typically corresponds to a different directed graph.}.
The adjacency matrix of an undirected random graph has to be symmetric, so its entries cannot be independent which makes the problem more difficult. 
We expect that our proof techniques will still give insights provided that appropriate modifications are made, but these modifications pose a nontrivial challenge; see \Cref{sec: FutureWork}.

\subsection{Results}
Given an Abelian group $H$ and a prime power $p^m$ we denote $H_{p^m} \de H/p^mH$. 
Then, the condition of \Cref{thm: Wang} is satisfied if and only if $\coker(\bW)_{p^2} \in  \{0,\bbZ/p\bbZ \}$  for every odd prime $p$ and $\coker(\bW)_{2^2} \cong (\bbZ/2\bbZ)^{\lfloor n/2\rfloor}$.   
To determine how frequently \Cref{thm: Wang} is applicable, it hence suffices to study the joint distribution of $\coker(\bW)_{p^2}$ over all primes $p$ when $\bX$ is the adjacency matrix of an undirected Erd\H{o}s--R\'enyi random graph and $b=e$.

In a directed and weighted setting, the following result gives a remarkably simple formula for the limiting marginal distribution for a single prime $p$: 
\begin{theorem}\label{thm: MainResult_Uniform}
    Fix a prime $p$ and an integer $m\geq 1$. 
    For every $n\geq 1$, let $\bX$ be a $\bbZ^{n\times n}$-valued random matrix with independent $\Unif\{0,1,\ldots,p^{m} - 1\}$-distributed entries and let $b\in \bbZ^{n}$ be a deterministic vector with $b\not\equiv 0 \bmod p$. \nopagebreak

    Fix some $\ell \geq 1$ and $0 = \lambda_0 \leq \lambda_1 \leq \cdots \leq \lambda_\ell \leq m$. 
    Let $i_0\de \#\{i\leq \ell : \lambda_i = m \}$ and denote $\delta_j \de \lambda_{\ell-j+1} - \lambda_{\ell -j}$.
    Then, as Abelian groups, 
    \begin{align*}
        \lim_{n\to \infty} \bbP\Bigl(\coker(\bW)_{p^m} \cong \bigoplus_{i=1}^\ell \frac{\bbZ}{p^{\lambda_i} \bbZ}\Bigr)  = \prod_{i=i_0}^\infty\Bigl(1-p^{-(i+1)}\Bigr)\prod_{j=1}^\ell p^{-j  \delta_j}. 
    \end{align*} 
\end{theorem}
\newpage 
Note in particular that \Cref{thm: MainResult_Uniform} predicts that $\coker(\bW)_{p^2}\in \{0,\bbZ/p\bbZ \}$ with nonvanishing probability, at least for random directed and weighted graphs.
For example, the limiting probabilities associated with the primes $p=3$, $5$, and $7$ are approximately $0.75$, $0.91$, and $0.96$, respectively.
We here emphasize odd primes because, while \Cref{thm: MainResult_Uniform} also applies when $p=2$, it is known that the distribution of $\coker(\bW)_{2^m}$ is very different for simple and non-simple graphs; see \Cref{sec: FutureWork}. 
It would hence be ill-advised to use directed graphs as a model for the condition $\coker(\bW)_{2^2} \cong (\bbZ/2\bbZ)^{\lfloor n/2 \rfloor}$.

The proof of \Cref{thm: MainResult_Uniform} relies on interpretable combinatorial arguments. 
The downside is that the distributional assumption on the entries of $\bX$ plays a crucial role which makes the approach unsuitable for the study of unweighted graphs.
Fortunately, a different proof approach allows us to study $\coker(\bW)$ in a general setting which also covers the case where $\bX$ has $\{0,1 \}$-valued entries. 
Additionally, this allows us to gain insight on the joint law across different primes and the result even applies to sparse graphs. 

For this more general setting, it turns out to be essential to interact with all canonical structure on $\coker(\bW)$. 
Equip $\bbZ^n$ with the $\bbZ[x]$-module structure defined by $xv \de \bX v$ and note that $\bW(\bbZ^n)$ is precisely the $\bbZ[x]$-submodule of $\bbZ^n$ generated by $b$. 
Hence, the quotient $\coker(\bW)$ is canonically equipped with the structure of a $\bbZ[x]$-module.
We require some additional notation. 
Let $Q(x)\in \bbZ[x]$ be a monic polynomial and consider a prime power $p^m$. 
Then, given a $\bbZ[x]$-module $N$, we define a quotient module by $N_{p^m, Q} \de N/(p^mN + Q(x)N)$.
We further abbreviate $R_{p^m,Q} \de \bbZ[x]/(p^m\bbZ[x] + Q(x)\bbZ[x])$. 

Fix a finite collection of prime numbers $\sP$ and consider a scalar $\alpha >0$. 
Then, a $\bbZ$-valued random variable $Y$ is said to be \emph{$\alpha$-balanced mod $\sP$} if for all $p \in \sP$ and $y \in \bbZ/p\bbZ$, 
\begin{align}
    \bbP(Y \equiv y \bmod p) \leq 1-\alpha.\nonumber    
\end{align}
Given a ring $R$, recall that $\Ext^{1}_R(N,M)$ denotes the set of extensions of an $R$-modules $N$ by an $R$-module $M$ \cite[p.77]{weibel1994introduction} and denote $\Aut_R(N)$ and $\Hom_R(N,M)$ for the sets of $R$-module automorphisms and homomorphisms, respectively.  
\begin{theorem}\label{thm: Sparse_MainResult}  
    Fix a finite set of primes $\sP$.   
    For every $n\geq 1$ let $\bX$ be a $\bbZ^{n\times n}$-valued random matrix with independent entries, not necessarily identically distributed, such that each entry $\bX_{i,j}$ is $\alpha_n$-balanced mod $\sP$. 
    Further, let $b\in \bbZ^{n}$ be deterministic with $b\not\equiv 0\bmod p$ for every $p\in \sP$.

    Assume that $\lim_{n\to \infty} n\alpha_n/\ln(n) = \infty$.
    Then, for every integer $m\geq 1$, monic polynomial $Q\in \bbZ[x]$ of degree $\geq 1$, and collection of finite $R_{p^m,Q}$-modules $N_{p^m,Q}$: 
    \begin{enumerate}[leftmargin = 2em, label = (\arabic*)]
        \item The quotients $\coker(\bW)_{p^m,Q}$ associated with different primes $p\in \sP$ are asymptotically independent. 
        More precisely, as $\bbZ[x]$-modules,   
        \begin{align*}
            \lim_{n\to \infty}\bbP\bigl(\forall p\in \sP:\coker(\bW)_{p^m,Q} \cong N_{p^m,Q} \bigr) = \prod_{p\in \sP}\mu_{p^m,Q}(N_{p^m,Q} )
        \end{align*} 
        for certain probability measures $\mu_{p^m,Q}$ supported on finite $R_{p^m,Q}$-modules. 
        \item Suppose that $Q(x) \equiv \prod_{i=1}^{r_p} Q_{i,p}(x)^{q_{i,p}} \bmod p$ is the unique factorization of $Q\bmod p$ into powers of distinct monic irreducible polynomials $Q_{i,p} \in \bbF_p[x]$.
        Let $d_{i,p} \de \deg Q_{i,p}$ denote the degree of $Q_{i,p}$. 
        Then, the measure $\mu_{p^m,Q}$ is given by the following identity:       
        \begin{align*}
            \mu_{p^m,Q}\bigl(N_{p^m,Q}\bigr) ={}&{} \frac{1}{\#N_{p^m,Q} \#\Aut_{R_{p^m,Q}}(N_{p^m,Q}) }\prod_{i = 1}^{r_p}\prod_{j=1}^\infty  \\ 
            &\ \times \Bigl(1-\frac{\#\Ext_{R_{p^m,Q}}^1\bigl(N_{p^m,Q}, \bbF_p[x]/(Q_{i,p}(x)\bbF_p[x]) \bigr)}{\#\Hom_{R_{p^m,Q}}\bigl(N_{p^m,Q},\bbF_p[x]/(Q_{i,p}(x)\bbF_p[x] \bigr)} p^{-(1+j)d_{i,p} } \Bigr).   
        \end{align*}     
    \end{enumerate}
\end{theorem} 
\newpage
So far as it pertains to the group structure, the conclusion of \Cref{thm: Sparse_MainResult} with $\sP = \{p \}$ is weaker than \Cref{thm: MainResult_Uniform}, but only slightly.
To study $\coker(\bW)_{p^m}$ itself, without the additional quotient by $Q(x)$, it would namely be sufficient to have a tightness condition stating that 
$ 
    \lim_{C\to \infty} \liminf_{n\to \infty} \bbP(\# \coker(\bW)_{p^m} \leq C) = 1. 
$
Such a condition would yield the limiting law as a $\bbZ[x]$-module\footnote{Indeed, note that the tightness condition would imply that $\coker(\bW)_{p^m,Q} \cong \coker(\bW)_{p^m}$ with high probability for $Q(x) \de \prod_{i\in \{1,\ldots,r \}}\prod_{j \in \{1,\ldots,r \}\setminus\{i\}} (t^i-t^j)$ with $r$ sufficiently large because a finite group can only admit finitely many distinct endomorphisms.} and the limiting law as a group then follows by summing over all $\bbZ[x]$-module structures on $\oplus_{i=1}^\ell \bbZ/p^{\lambda_i}\bbZ$.

We are not aware of a direct proof that the aforementioned sum recovers the formula in \Cref{thm: MainResult_Uniform}, but this would follow indirectly since \Cref{thm: Sparse_MainResult} also applies to a matrix with uniform entries as in \Cref{thm: MainResult_Uniform}. 
So, if the tightness condition holds, then the distribution of $\coker(\bW)_{p^m}$ converges to the same limit as in \Cref{thm: MainResult_Uniform} and one has asymptotic independence for any fixed finite set of primes. 
Hence, additionally assuming that the restriction to finite sets of primes can also be removed, we are led to the following: 
\begin{conjecture}\label{conj: OddPrimeWellBehaved}
    For every $n\geq 1$, let $\bX$ be a $\{0,1 \}^{n\times n}$-valued random matrix with independent entries and let $b = e$ be the all-ones vector.
    Assume that there exists a sequence $\alpha_n$ satisfying $\lim_{n\to \infty} n\alpha_n/\ln(n)  = \infty$ such that $\bbP(\bX_{i,j} = 0), \bbP(\bX_{i,j} =1)\leq 1 -\alpha_n$ for each entry. 
    Then, 
    \begin{align} 
        \lim_{n\to \infty}{}&{} \bbP\bigl(\coker(\bW)_{p^2} \in \{0,\bbZ/p\bbZ \}\text{ for every odd prime p}\bigr)=\negsp \negsp \prod_{\text{odd primes }p}\negsp \negsp
        (1+p^{-1})
        \prod_{i=0}^{\infty} 
        (1-p^{-(i+1)}). \nonumber  
    \end{align} 
\end{conjecture}

\begin{remark}\label{rem: Sparsity}
    The restriction that $\alpha_n \gg \ln(n)/n$ in \Cref{thm: Sparse_MainResult}, which limits how sparse the matrices are allowed to be, is close to optimal.
    The conclusion of \Cref{thm: Sparse_MainResult} namely has to fail when $\bX$ has independent entries satisfying $\bbP(\bX_{i,j} = 0) \geq 1 - (1-\varepsilon)\ln(n)/n$ for $\varepsilon>0$.

    Indeed, the coupon collector theorem then implies that $\bX$ has many rows equal to zero so that $\rank(\bX) \leq n-2$ with high probability.
    The latter implies that $\bX(\bbF_p^n) + \bbF_p b \neq \bbF_p^n$.   
    Hence, since $\coker(\bW)_{p,x} \cong \bbZ^n/( \bW(\bbZ^n) + p\bbZ^n + x\bbZ^n)$ is isomorphic to $\bbF_p^n/(\bX(\bbF_p^n) + \bbF_p b)$, it would follow that $\lim_{n\to \infty}\bbP(\coker(\bW)_{p,x} = \{0 \}) = 0$.
    This is incompatible with the conclusion of \Cref{thm: Sparse_MainResult} since $\mu_{p,x}(\{0\}) = \prod_{j=1}^\infty(1-p^{-1 - j}) \neq 0$.      
\end{remark}
\begin{remark}
    The limiting distribution of the $\bbZ[x]$-module $\coker(Q(\bX))$ with $Q(x)$ a fixed polynomial was recently studied by Cheong and Yu \cite{cheong2023cokernel}. 
    Interestingly, the formulas found in \Cref{thm: Sparse_MainResult} and \cite[Theorem 1.3]{cheong2023cokernel} bear a close resemblance, although they are not identical.  
    This resemblance is not entirely suprising since the studied objects can be related. 
    Indeed, note that $\coker(Q(\bX))_{p^m} \cong \bbZ^n / (Q(x) \bbZ^n + p^m\bbZ^n)$ whereas $\coker(\bW)_{p^m,Q}\cong \bbZ^n/(Q(x)\bbZ^n + p^m\bbZ^n + \bbZ[x]b)$.

    One can however not directly recover \Cref{thm: Sparse_MainResult} from \cite{cheong2023cokernel}.
    For one thing, \cite{cheong2023cokernel} only considers the case $\sP = \{p\}$ and does not allow sparse settings where $\alpha_n \to 0$.  
    Further, $\coker(\bW)_{p^m, Q}$ not only depends on the isomorphism class of $\coker(Q(\bX))_{p^m}$ but also on how the reduction of $b$ lies in $\coker(Q(\bX))_{p^m}$, which is information we do not have access to.    
\end{remark}
\subsection{Proof techniques}\label{sec: ProofTechniques}
The proof of \Cref{thm: MainResult_Uniform} relies on an analysis of the sequence of random vectors $(\bX^{t-1}b)_{t=1}^n$, viewed as a stochastic process. 
More precisely, we show in \Cref{cor: Ut_Coker} that the group structure of $\coker(\bW)_{p^m}$ can be computed in terms of the sequence of random variables $0 = U_1 \leq U_2 \leq \ldots \leq U_n \leq \infty$ defined by  
\begin{align} 
    U_t \de \sup\bigl\{j\geq 0: \bX^{t-1}b \in \spn_{\bbZ}(\bX^{i-1}b: 1\leq  i \leq t-1 ) + p^j\bbZ^n \bigr\}.  \label{eq:KnownHen}
\end{align}
Here, given a ring $R$ and $v_1,\ldots,v_t \in R^n$ we write $\spn_{R}(v_1,\ldots,v_t) \de \{\sum_{i=1}^t c_i v_i: c_i \in R \}$. 
The remaining difficulty is then to study the joint law of the $U_j$. 
The distributional assumption in \Cref{thm: MainResult_Uniform} plays an important role for the latter task: the independence and equidistribution of the entries implies that $\bX$ induces a uniform random endomorphism of $(\bbZ/p^m\bbZ)^n$ which can be used to establish Markovian dynamics for $\min\{U_t, m\}$; see \Cref{lem: Dynamics}. 

This proof yields a direct combinatorial interpretation for the formula in \Cref{thm: MainResult_Uniform}.
Given $U_t$, a counting argument implies that the probability that $U_{t+1} \geq U_t + \delta$ is $p^{-\delta(n-t)}$ for any $\delta$ satisfying $U_t + \delta \leq m$. 
Taking $j=n-t$ then explains the factors of the form $p^{-j\delta_j}$ in \Cref{thm: MainResult_Uniform}. 
Factors of the form $1-p^{-j}$ arise when we additionally have to enforce that $U_{t+1} \leq U_t + \delta$.     
A further benefit of the approach is that it can also be used to establish the law of $\coker(\bW)_{p^m}$ for finite $n$; see \Cref{prop: UnifNonasymptotic}. 

The proof of \Cref{thm: Sparse_MainResult} relies on more sophisticated techniques. 
In particular, we employ the \emph{category-theoretic moment method}. 
In classical probability theory, the moment method allows one to establish convergence in distribution of a sequence of $\bbR$-valued random variables $(Y_i)_{i=1}^\infty$ by showing that the moments $\bbE[Y_i^n]$ converge to those of the desired limiting law provided that some mild conditions are satisfied.
Results of Sawin and Wood \cite{sawin2022moment} similarly allow one to establish limiting laws for random algebraic objects such as groups and modules by showing that category-theoretic moments converge.
Here, given a ring $R$ and a deterministic $R$-module $N$, the \emph{$N$-moment}\footnote{An explanation for the terminology \emph{moment} may be found in \cite[Section 3.3]{clancy2015cohen}.} of a random $R$-module $Y$ is given by $\bbE[\#\Sur_{R}(Y,N)]$ where $\Sur_R(Y,N)$ denotes the set of surjective $R$-module morphisms from $Y$ to $N$.

The main challenge is hence to estimate the $N$-moments of the random $\bbZ[x]$-module $\coker(\bW)$. 
To this end, we employ a strategy developed by Wood \cite{wood2017distribution,wood2019random} and Nguyen and Wood \cite{nguyen2022random} for the estimation of moments of random algebraic objects associated with random matrices with $\alpha$-balanced entries.
A key difference between our setting and the one in \cite{wood2017distribution,wood2019random,nguyen2022random} is that we have little control over the joint law of the entries $\bW_{i,j}$ of our matrix of interest since these are nontrivial algebraic combinations of the entries of $\bX$.  
This is why it is essential to view $\coker(\bW)$ as a $\bbZ[x]$-module, not only as a group.
That is, the $\bbZ[x]$-module-theoretic viewpoint allows us to untangle the algebraic dependencies and hence estimate the category-theoretic moments; see \eqref{eq: SurMomentToProbSum} and the subsequent remarks. 

The advantage of the category-theoretic approach is that it is robust, as is demonstrated by the general distributional assumptions in \Cref{thm: Sparse_MainResult}. 
For comparison, the proof approach for \Cref{thm: MainResult_Uniform} is highly non-robust. 
Indeed, as mentioned above, that proof uses that $\bX$ induces a uniform random endomorphism of $(\bbZ/p^m\bbZ)^n$; a property which is only satisfied when $\bX \bmod p^m$ has independent and uniformly distributed entries.
The robustness of the category-theoretic approach makes us hopeful that we will be able to generalize it in future work, although this remains a nontrivial task; see \Cref{sec: FutureWork}.

\subsection{Structure of this paper}
The proof of \Cref{thm: MainResult_Uniform} is given in \Cref{sec: ProofUniform}. 
We there also give a non-asymptotic variant of \Cref{thm: MainResult_Uniform}. 
The proof of \Cref{thm: Sparse_MainResult} is given in \Cref{sec: ProofSparse}.
Directions for future work are outlined in \Cref{sec: FutureWork}
\section{Proof of \texorpdfstring{\Cref{thm: MainResult_Uniform}}{Theorem}}\label{sec: ProofUniform}
Throughout this section, we fix a prime $p$ and a vector $b\in \bbZ^n$ with $b\not\equiv 0 \bmod p$. 
Recall from \Cref{sec: ProofTechniques} that the proof has two main ingredients.
The first ingredient is \Cref{cor: Ut_Coker} which shows that the group structure of $\coker(\bW)_{p^m}$ can be computed in terms of $U_1,\ldots,U_n$. 
The second ingredient is \Cref{lem: Dynamics} which concerns the joint law of the $U_t$ when $\bX$ is random as in \Cref{thm: MainResult_Uniform}.
We combine these results to establish \Cref{thm: MainResult_Uniform} in \Cref{sec: LawCokerpm}. 
\subsection{Computing \texorpdfstring{$\coker(\bW)_{p^m}$}{coker(K)} in terms of \texorpdfstring{$U_1,\ldots,U_n$}{U1,...,Un}}
Recall the definition of $U_t$ from \eqref{eq:KnownHen}. 
Fix an integer $m\geq 1$ and abbreviate $R \de \bbZ/p^m\bbZ$.
The following lemma then produces an $R$-module basis for $R^n$ which is well-adapted to the computation of $\coker(\bW)_{p^m}$.  
\begin{lemma}\label{lem: ModuleBasis}
    Write $\Lambda_t \de \min\{U_t,m \}$.
    Then, there exist $v_1,\ldots,v_n \in R^n$  such that for every $t \leq n$ the following properties are satisfied:  
    \begin{enumerate}[leftmargin = 2em, label = (\roman*)]
        \item \label{item: rank}The reduction of the matrix $(v_1,\ldots,v_t)$ modulo $p$ has rank $t$ over $\bbF_p$. 
        \item \label{item: v_span_X_span}It holds that $ \spn_{R}(\bX^{i-1} b \bmod p^m:1\leq i\leq t) = \spn_{R}(p^{\Lambda_i}v_i: 1 \leq i\leq t)$.        
    \end{enumerate}    
\end{lemma}

\begin{proof}
    We proceed by induction on $t$. 
    If $t = 1$, then $U_1 = 0$ so $v_1 \equiv b \bmod p^m$ satisfies both properties. 
    Now suppose that $t>1$ and assume that there exist $v_1,\ldots,v_{t-1}$ such that both properties are satisfied. 
    We prove the existence of some $v_{t}$. 

    First, suppose that $U_{t} \geq m$. 
    Then, the definition \eqref{eq:KnownHen} yields $\spn_{R}(\bX^{i-1} b \bmod p^m:i\leq t) = \spn_{R}(\bX^{i-1} b \bmod p^m:i\leq t-1)$ and $p^{\Lambda_{t}}v = p^{m}v = 0$ for every $v\in R^n$. 
    Consequently, due to the induction hypothesis, both properties are satisfied if we let $v_{t}\in R^n$ be an arbitrary vector which is not in $\spn_{R}(v_1,\ldots,v_{t-1}) + pR^n$. 
    Such a vector exists because $t -1 < n$.
    
    Now suppose that $U_{t} <m$. 
    By definition of $U_{t}$ there then exist $w \in p^{U_{t}}R^n$ and $r \in \spn_{R}(\bX^{i-1}b\bmod p^m: i\leq  t-1)$ such that $\bX^{t-1}b \equiv w + r \bmod p^m$. 
    Pick some $v_{t}\in R^n$ with $p^{U_{t}}v_{t}= w$.   
    Then, due to the induction hypothesis, \cref{item: v_span_X_span} is satisfied and \cref{item: rank} is equivalent to the statement that $v_{t} \not\in \spn_R(v_1,\ldots,v_{t-1}) + pR^n$.  
    The latter statement is true. 
    Indeed, if not, then $w \in \spn_R(p^{U_{t}}v_1,\ldots, p^{U_{t}}v_{t-1})\allowbreak + p^{U_{t} + 1}R^n$. 
    Then, considering that $\spn_R(p^{U_{t}}v_i: i\leq t-1) \subseteq \spn_R(\bX^{i-1}b \bmod p^m:i\leq t-1)$ by \cref{item: v_span_X_span} of the induction hypothesis and the fact that the $U_i$ are nondecreasing, it follows from $\bX^{t-1}b \equiv w + r \bmod p^m$ that $\bX^{t-1}b \in \spn_\bbZ(\bX^{i-1}b :i\leq t-1) + p^{U_{t}+1}\bbZ^n$ contradicting the maximality of $U_{t}$ in \eqref{eq:KnownHen}. 
    This shows that both properties are satisfied. 
\end{proof}

\begin{corollary}\label{cor: Ut_Coker}
    Adopt the notation of \Cref{lem: ModuleBasis}. 
    Then, $\coker(\bW)_{p^m} \cong \oplus_{t=1}^n \bbZ/p^{\Lambda_t}\bbZ$. 
\end{corollary}
\begin{proof}
    The case with $t=n$ in \cref{item: rank} of \Cref{lem: ModuleBasis} implies that the vectors $v_1,\ldots,v_n \in R^n$ determine an $R$-module basis for $R^n$.
    Further, \cref{item: v_span_X_span} yields that $\bW(R^n) = \spn_{R}(p^{\Lambda_1}v_1,\allowbreak\ldots,p^{\Lambda_n}v_n)$.
    The claim is hence immediate since $\coker(\bW)_{p^m} \cong R^n/\bW(R^n)$. 
\end{proof}

\subsection{Markovian dynamics}
Given a matrix $\bM$, abbreviate $\rank_p(\bM)$ for the rank of $\bM \bmod p$ over $\bbF_p$.
Recall that $R = \bbZ/p^m\bbZ$. 
The following lemma provides a partial converse for \Cref{lem: ModuleBasis} in the case $U_t<m$:
\begin{lemma}\label{lem: UiCharacterization}
    Consider some $t\leq n$ and integers $0 = u_1\leq u_2 \leq \ldots \leq u_t < m$. 
    Then, it holds that $U_i = u_i$ for every $i\leq t$ if and only if there exist $v_1,\ldots,v_t \in R^n$ with $\rank_p(v_1,\ldots,v_t) = t$ such that for every $i\leq t$ one has 
    $\spn_{R}(\bX^{j-1} b \bmod p^m:j\leq i) =\spn_{R}(p^{u_j}v_j: j\leq i)$. 
\end{lemma}
\begin{proof}
    If $U_i = u_i$ for every $i\leq t$ then the existence of $v_1,\ldots,v_t$ with the claimed properties follows from \Cref{lem: ModuleBasis}. 
    Conversely, assume that such $v_1,\ldots,v_t$ exist.
    Then, for every $i\leq t$, 
    \begin{align} 
        \spn_R(\bX^{j-1}b \bmod p^m: j\leq i) &= \spn_R(p^{u_j}v_j:j<i) + \spn_R(p^{u_i}v_i) \label{eq:IcyDew}\\ 
        &\subseteq \spn_R(\bX^{j-1}b \bmod p^m: j< i) + p^{u_i}R^n.\nonumber     
    \end{align}
    Considering that $R = \bbZ/p^m\bbZ$ with $m \geq u_i$, it follows that $\bX^{i-1}b \in \spn_\bbZ(\bX^{j-1}b: j< i) + p^{u_i}\bbZ^n$. 
    The definition \eqref{eq:KnownHen} hence yields $U_i \geq u_i$. 
    On the other hand, since $\rank_p(v_1,\ldots,v_t) = t$ and $u_i < m$, we have $p^{u_i}v_i \not\in \spn_R(p^{u_j}v_j:j < i) + p^{u_i + 1}R^n$.
    Hence,    
    \begin{align} 
        \spn_R(p^{u_j}v_j:j<i) + \spn_R(p^{u_i}v_i) 
        &\not\subseteq \spn_R(p^{u_j}v_j:j<i) + p^{u_i + 1}R^n\\ 
        &= \spn_R(\bX^{j-1}b \bmod p^m: j< i) + p^{u_i + 1}R^n.\nonumber 
    \end{align} 
    Given the equality in \eqref{eq:IcyDew} and the fact that $u_i+1 \leq m$, this implies that $\bX^{i-1}b \not\in \spn_{\bbZ}(\bX^{j-1}b:j<i) + p^{u_i+1}\bbZ^n$. 
    This means that $U_i\leq u_i$. 
    Combine the inequalities $U_i \geq u_i$ and $U_i \leq u_i$ to conclude the proof.    
\end{proof}

\begin{lemma}\label{lem: Dynamics}
    Assume that $\bX$ has independent and $\Unif\{0,1,\ldots,p^{m}-1 \}$-distributed entries. 
    Consider some $t\leq n-1$. 
    Then, for every $0= u_1\leq \ldots\leq u_t< m$ and $u_{t+1} \leq m$,     
    \begin{align} 
        \bbP\bigl(U_{t+1} \geq u_{t+1} \mid U_{i} =u_i, \, \forall i \in \{1,\ldots,t\}\bigr)  = p^{-(n-t)(u_{t+1} - u_t)}.\label{eq:ZenBeetle}
    \end{align}
    In particular, if additionally $u_{t+1} < m$, 
    \begin{align} 
        \bbP\bigl(U_{t+1} = u_{t+1} \mid U_{i} =u_i, \, \forall i \in \{1,\ldots,t\}\bigr)  = p^{-(n-t)(u_{t+1} - u_t)} \bigl(1 - p^{-(n-t)} \bigr). \label{eq:FancyBall}
    \end{align}
\end{lemma}
\begin{proof}
    Two ordered sets of vectors $v_1,\ldots,v_t \in R^n$ and $w_1,\ldots,w_t\in R^n$ are said to be \emph{equivalent} if $\spn_{R}(p^{u_j}v_j:j\leq i) = \spn_{R}(p^{u_j}w_j:j\leq i)$ for every $i \leq t$.  
    
    \Cref{lem: UiCharacterization} implies that the event $\{U_i = u_i:\forall i\leq t \}$ can be written as a union of events of the form $\{\spn_{R}(\bX^{j-1} b \bmod p^m:j\leq i) = \spn_{R}(p^{u_j}v_j:j\leq i),\, \forall i\leq t\}$ indexed by vectors $v_1,\ldots,v_t\in R^n$ with $\rank_p(v_1,\ldots,v_t) = t$ and $\spn_R(v_1) = \spn_R(b\bmod p^m)$.
    Two such events are equal if the corresponding sets of vectors are equivalent and mutually exclusive otherwise. 
    Hence, by conditioning on the equivalence class, \eqref{eq:ZenBeetle} follows if we show that for every $v_1,\ldots,v_t \in R^n$ with $\rank_p(v_1,\ldots,v_t) = t$ and $\spn_R(v_1) = \spn_R(b\bmod p^m)$, 
    \begin{align} 
        \bbP\bigl( U_{t+1} \geq u_{t+1} \mid \spn_R(\bX^{j-1}b \bmod p^m: j\leq i) = \spn_{R}({}&{}p^{u_j}v_j:j\leq i)  ,\,  \forall i\leq t \bigr) \label{eq:OddElf}\\ 
        &\ = p^{-(n-t)(u_{t+1}-u_t)}.\nonumber
    \end{align}
    Fix $v_1,\ldots,v_t$ and let $E\de \{\spn_R(\bX^{j-1}b \bmod p^m: j\leq i) = \spn_{R}(p^{u_j}v_j:j\leq i)  ,\,  \forall i\leq t\}$ denote the event in the condition of \eqref{eq:OddElf}.
        
    It follows from the definition that $U_{t+1} \geq u_{t+1}$ if and only if $\bX(\spn_\bbZ(\bX^{i-1}b:i\leq t))\subseteq \spn_\bbZ(\bX^{i-1}b:i\leq t) + p^{u_{t+1}}\bbZ^n$. 
    Hence, conditional on $E$, one has $U_{t+1} \geq u_{t+1}$ if and only if $\bX(p^{u_t}v_t) \in \spn_R(p^{u_j}v_j: j\leq t) +p^{u_{t+1}}R^n$.
    Considering that $\rank_p(v_1,\ldots,v_t)=t$ and that the $u_i$ are nondecreasing, the latter occurs if and only if $\bX(v_t) \in \spn_R(v_1,\ldots,v_t) + p^{u_{t+1} - u_t}R^n$. 
    Hence,   
    \begin{align} 
        \bbP\bigl( U_{t+1} \geq u_{t+1} \mid E\bigr) = \bbP\bigl(\bX(v_t) \in \spn_R(v_1,\ldots,v_t) + p^{u_{t+1} - u_t}R^n \mid E\bigr).\label{eq:HauntedDad}  
    \end{align} 
    The assumption $\spn_R(v_1) = \spn_R(b\bmod p^m)$ implies that the event $\spn_R(\bX^{j-1} b \bmod p^m:j\leq 2) = \spn_R(p^{u_j}v_1:j\leq 2)$ only depends on $\bX(v_1)$.
    Similarly, continuing in an inductive fashion, the event $E$ only depends on $\bX(v_1),\ldots,\bX(v_{t-1})$. 
    Hence, by the law of total probability, 
    \begin{align} 
        \bbP\bigl(\bX(&v_t) \in \spn_R(v_1,\ldots,v_t) + p^{u_{t+1} - u_t}R^n \mid E\bigr)\label{eq:NormalBat} \\ 
        &= \bbE\Bigl[\bbP\bigl(\bX(v_t) \in \spn_R(v_1,\ldots,v_t) + p^{u_{t+1} - u_t}R^n \mid \bX(v_1),\ldots,\bX(v_{t-1})\bigr)\ \big\vert \ E\  \Bigr].\nonumber 
    \end{align}

    Recall that the entries of $\bX$ are independent and $\Unif\{0,1,\ldots,p^{m}-1 \}$-distributed. 
    This implies that $\bX$ induces a uniform random endomorphism of $R^n$. 
    Hence, since it was assumed that $\rank_p(v_1,\ldots,v_t) = t$, it holds that $\bX(v_t)$ has a uniform distribution on $R^n$ and is independent of $\bX(v_1),\ldots,\bX(v_{t-1})$. 
    Consequently, a counting argument yields that    
    \begin{align} 
        \bbP\bigl(\bX(v_t) \in \spn_R(v_1,\ldots,v_t) + p^{u_{t+1} - u_t}R^n \mid \bX(v_1),\ldots,\bX(v_{t-1})\bigr) = p^{-(n - t)(u_{t+1} - u_t)}. \label{eq:VividFog}
    \end{align}
    Combine \eqref{eq:HauntedDad}--\eqref{eq:VividFog} to establish \eqref{eq:OddElf}. 
    This proves \eqref{eq:ZenBeetle}. 
    Further, \eqref{eq:FancyBall} is an immediate consequence of \eqref{eq:ZenBeetle} since $U_{t+1} = u_{t+1}$ if and only if $U_{t+1} \geq u_{t+1}$ and $U_{t+1} < u_{t+1} + 1$.
\end{proof}
\subsection{The law of \texorpdfstring{$\coker(\bW)_{p^m}$}{coker(W)pm}}\label{sec: LawCokerpm}
It now only remains to combine \Cref{cor: Ut_Coker} and \Cref{lem: Dynamics}. 
This allows us to also determine the law of $\coker(\bW)_{p^m}$ when $n$ is finite:
\begin{proposition}\label{prop: UnifNonasymptotic} 
    Fix some $n\geq 1$, let $\bX$ be a $\bbZ^{n\times n}$-valued random matrix with independent $\Unif\{0,1,\ldots,\allowbreak p^{m} - 1\}$-distributed entries and let $b\in \bbZ^{n}$ be a deterministic vector with $b\not\equiv 0 \bmod p$. 
    Fix an integer $0\leq i_0 \leq n-1$. 

    Pick integers $0 = \lambda_1 \leq \lambda_2 \leq \ldots \leq \lambda_n \leq m$ and denote $\delta_i = \lambda_{n-i+1} - \lambda_{n-i}$. 
    Assume that $\lambda_{i} < m$ if and only if $i \leq n-i_0$.  
    Then, as Abelian groups, 
    \begin{align*}
        \bbP\Bigl(\coker(\bW)_{p^m} \cong \bigoplus_{i=1}^n \frac{\bbZ}{p^{\lambda_i} \bbZ}\Bigl)  = \prod_{i=i_0}^{n-2} \Bigl(1-p^{-(i+1)}\Bigr)\prod_{j=1}^{n-1} p^{-j  \delta_j}. 
    \end{align*} 
\end{proposition}
\begin{proof}
    By \Cref{cor: Ut_Coker} and the assumption that $\lambda_i < m$ for every $i\leq n-i_0$ and $\lambda_{i} = m$ for every $i > n-i_0$,    
    \begin{align} 
        \bbP\bigl(\coker(\bW)_{p^m} \cong {}&{}\oplus_{i=1}^n \bbZ/p^{\lambda_i}\bbZ\bigl) = \bbP\bigl(U_i = \lambda_i,\, \forall i \in \{1,\ldots,n-i_0 \}\bigr)  \label{eq:AvidLeaf} \\ 
        &\quad \times  \bbP\bigl(U_i\geq m,\, \forall i \in \{n-i_0 +1,\ldots,n \} \mid U_i = \lambda_i,\, \forall i\in \{1,\ldots,n-i_0 \}\bigr). \nonumber 
    \end{align} 
    Recall that $\delta_{i} = \lambda_{n-i+1} - \lambda_{n-i}$. 
    Hence, using that $U_1 = 0$ together with \eqref{eq:FancyBall} from \Cref{lem: Dynamics} which is applicable due to the assumption that $\lambda_i < m$ for every $i \leq n-i_0$,
    \begin{align} 
        \bbP\bigl(U_i = \lambda_i,\, \forall  i \in \{1,\ldots,n-i_0 \} \bigr)
        &=\prod_{i=i_0}^{n-2} \bbP(U_{n-i} = U_{n-i-1} + \delta_{i+1} \mid U_j = \lambda_j, \, \forall j< n-i) \nonumber\\ 
        &=\prod_{i=i_0}^{n-2} (1-p^{-(i+1)}) p^{-(i+1)\delta_{i+1}}.  \label{eq:NewWisp}
    \end{align}
    If $i_0 = 0$, then the second probability in \eqref{eq:AvidLeaf} is equal to one since there is no $i$ satisfying $n + 1\leq i \leq n$.
    In this case, the combination of \eqref{eq:AvidLeaf} and \eqref{eq:NewWisp} concludes the proof.

    Now suppose that $i_0 >0$.    
    Then, it holds that $U_i \geq m$ for all $i> n-i_0$ if and only if $U_{n-i_0 + 1} \geq m$. 
    Hence, using \eqref{eq:ZenBeetle} from \Cref{lem: Dynamics} and recalling that $m = \lambda_{n-i_0} + \delta_{i_0}$,  
    \begin{align} 
        \bbP\bigl(U_i\geq m,\, \forall i \in \{n-i_0 +1,\ldots,n \} {}&{}\mid  U_i = \lambda_i,\, \forall i\in \{1,\ldots,n-i_0 \}\bigr)\label{eq:NewSquid}\\
        &= \bbP(U_{n-i_0 +1} \geq m \mid   U_i = \lambda_i,\, \forall i\in \{1,\ldots,n-i_0 \} \bigr)\nonumber\\ 
        &=p^{-i_0\delta_{i_0}}.\nonumber      
    \end{align}    
    Remark $p^{-i_0\delta_{i_0}}=\prod_{i=1}^{i_0} p^{-i \delta_i}$ since the assumption that $\lambda_i =  m  = \lambda_{i+1}$ for all $i>n-i_0$ ensures that $\delta_{i} = 0$ for every $i<i_0$. 
    The combination of \eqref{eq:AvidLeaf}--\eqref{eq:NewSquid} hence concludes the proof. 
\end{proof}

\begin{proof}[Proof of \texorpdfstring{\Cref{thm: MainResult_Uniform}}{Theorem}]
    Let $\tilde{\lambda}_i \de 0$ for $i \in \{1,\ldots, n-\ell \}$ and let $\tilde{\lambda}_{i} \de \lambda_{i -(n-\ell)}$ for $i\geq n-\ell +1$. 
    The result then follows by considering the probability that $\coker(\bW)_{p^m} \cong \oplus_{i=1}^n\bbZ/p^{\tilde{\lambda}_i}\bbZ$ in \Cref{prop: UnifNonasymptotic} and taking the limit $n\to \infty$.      
\end{proof}

\section{Proof of \texorpdfstring{\Cref{thm: Sparse_MainResult}}{Theorem}}\label{sec: ProofSparse}
We establish a more general result than \Cref{thm: Sparse_MainResult} and study the limiting law of the $\bbZ[x]$-module $\coker(\tilde{\bW})$ where $\tilde{\bW} \de (\bX^{j-1}\bB)_{j=1}^n$ is the $n\times nk$ matrix associated to a deterministic $n\times k$ matrix $\bB$ for some fixed $k$. 
For future reference, let us here state all relevant assumptions:
\begin{enumerate}[leftmargin=3em, label = (A\arabic*)]
    \item \label{a: X}For every $n\geq 1$ let $\bX$ be a $\bbZ^{n\times n}$-valued random matrix with independent entries such that each entry is $\alpha_n$-balanced mod $\sP$.
    \item \label{a: B} Fix some $k\geq 1$.
    For every $n\geq k$ let $\bB\in \bbZ^{n\times k}$ be a deterministic matrix such that $\bB \bmod p$ has rank $k$ over $\bbF_p$ for every $p\in \sP$.
    We denote $\tilde{\bW} \de (\bX^{j-1}\bB)_{j=1}^n$ and write $\coker(\tilde{\bW}) \de \bbZ^n/\tilde{\bW}(\bbZ^{nk})$. 
    \item \label{a: alpha_n}Assume that $\lim_{n\to \infty}n\alpha_n/\ln(n)= \infty$.  
\end{enumerate} 
The desired result, describing the limiting distribution of $\coker(\tilde{\bW})_{p^m,Q}$ under these assumptions, is given in \Cref{prop: InvertMomentSawinWood}.  

As was outlined in \Cref{sec: ProofTechniques}, we rely on the category-theoretic moment method. 
The main ingredient required for the proof is correspondingly an estimate on the moments of $\coker(\tilde{\bW})$:  
\begin{proposition}\label{prop: LimitingMoments}
    Adopt assumptions \ref{a: X} to \ref{a: alpha_n}. 
    Then, for every finite $\bbZ[x]$-module $N$ such that all prime divisors of $\#N$ are in $\sP$,   
    \begin{align}
        \lim_{n\to \infty}\bbE[\#\Sur_{\bbZ[x]}(\coker(\tilde{\bW}), N)] = (\# N)^{-k}.\label{eq: NMoment} 
    \end{align}
\end{proposition}
We prove \Cref{prop: LimitingMoments} in \Cref{sec: LimitingMoments} and then use a general-purpose result of Sawin and Wood \cite[Lemma 6.3]{sawin2022moment} to solve the associated moment problem in \Cref{sec: SolveMoment}.    
\begin{remark}
    The assumption in \Cref{prop: LimitingMoments} that all prime divisors of $\#N$ are in $\sP$ can not be removed. 
    Indeed, recall that \ref{a: X} and \ref{a: B} only make assumptions regarding $\rank_p(\bB)$ and the balanced nature of the entries of $\bX$ at primes $p \in \sP$.  

    So, for instance, at $p\notin\sP$ it could occur that $\bbP(\bX \equiv 0 \bmod p) = 1$ and $\bB = 0\bmod p$ in which case $\bbE[\# \Sur_{\bbZ[x]}(\coker(\tilde{\bW}), N)] = p^n -1$ with $N = \bbF_p[x]/x\bbF_{p}[x]$.
    In particular, it then holds that $\lim_{n\to \infty}\bbE[\# \Sur_{\bbZ[x]}(\coker(\tilde{\bW}), N)] = \infty$ which is incompatible with the conclusion of \Cref{prop: LimitingMoments}.     
\end{remark}
\begin{remark}
    There is a sense in which $\coker(\tilde{\bW})$ is a fairly natural random algebraic object to study. 
    Note that $\bW(\bbZ^{nk})$ is exactly the $\bbZ[x]$-submodule of $\bbZ^n$ generated by the columns of $\bB$. 
    Hence, introducing formal symbols $e_1,\ldots,e_n$, we have 
    \begin{align} 
        \coker(\tilde{\bW}) \cong \Bigl(e_1,\ldots,e_n: xe_j = \sum_{i=1}^n \bX_{i,j} e_i,\, \sum_{i=1}^n \bB_{i,r} e_i = 0,\, \forall i\leq n,\, \forall r\leq k  \Bigr)  
    \end{align}    
    as $\bbZ[x]$-modules. 
    So, $\coker(\tilde{\bW})$ corresponds to the finitely presented $\bbZ[x]$-module which is found when one considers $n$ generators, imposes a random action for $x$ specified by $\bX$, and imposes $k\geq 1$ additional deterministic constraints specified by $\bB$.  
\end{remark}

\subsection{Computing the limiting \texorpdfstring{$N$}{N}-moments}\label{sec: LimitingMoments}
Let $N$ be a deterministic $\bbZ[x]$-module such that all prime divisors of $\#N$ are in $\sP$.
We may consider $\bX$ as a random element of $\Hom_{\bbZ}(\bbZ^n,\bbZ^n)$, consider $\bB$ as a deterministic element of $\Hom_{\bbZ}(\bbZ^k,\bbZ^n)$, and consider $x$ as inducing an element of $\Hom_{\bbZ}(N,N)$.
Here, note that $\Hom_{\bbZ}(\cdot,\cdot)$ simply returns the set of group morphisms since a $\bbZ$-module and an Abelian group are the same thing. 

Now observe that a morphism of Abelian groups $F:\bbZ^n \to N$ descends to a morphism of $\bbZ[x]$-modules $\overline{F}:\coker(\tilde{\bW}) \to N$ if and only if the compositions of $F$ with $\bB,\bX$, and $x$ satisfy $F\bB = 0$ and $F\bX = xF$.
Moreover, every $\bbZ[x]$-module morphism $\overline{F}:\coker(\tilde{\bW}) \to N$ arises from some unique $F:\bbZ^n\to N$ in this fashion. 
Consequently, since surjectivity is conserved, 
\begin{align}
    \bbE[\# \Sur_{\bbZ[x]}(\coker(\tilde{\bW}), N)] &=
    \sum_{F\in \Sur_{\bbZ}(\bbZ^n,N): F\bB = 0} \bbP(F\bX = xF).\label{eq: SurMomentToProbSum} 
\end{align}
Let us here emphasize that, while \eqref{eq: SurMomentToProbSum} was relatively easy to prove, the simplification which this step offers is significant.
Indeed, observe that the joint law of the entries of $\tilde{\bW}$ is not easy to understand since these entries are nontrivial algebraic combinations of the entries of $\bX$ and $\bB$.  
On the other hand, $F\bX = xF$ is a linear equation in terms of $\bX$ and hence fairly explicit.

The strategy which we use to estimate the $N$-moments from here on is as follows.
We show that there are approximately $(\#N)^{n - k}$ surjections $F:\bbZ^n\to N$ with $F\bB = 0$ in \Cref{sec: NumSurB0}. 
Subsequently, we show that $\bbP(F\bX = xF) \approx (\# N)^{-n}$ for most terms in \eqref{eq: SurMomentToProbSum} in \Cref{sec: CodesEstimate}, and we show that the remaining terms give a negligible contribution in \Cref{sec: NonCodeEstimate}.
Finally, we combine these ingredients to conclude the proof of \Cref{prop: LimitingMoments} in \Cref{sec: CombineEstimates}.  

\subsubsection{Estimate on the number of surjections satisfying \texorpdfstring{$F\bB = 0$}{FB=0}}\label{sec: NumSurB0}
The \emph{exponent} of a finite Abelian group $G$ is the smallest positive integer $\exp(G)\geq 1$ such that $\exp(G)G = 0$. 
Note that $p\mid \exp(G)$ if and only if $p\mid \# G$.  
\begin{lemma}\label{lem: NumSubFB0}
    Let $G$ be a finite Abelian group and let $\bB \in \Hom_{\bbZ}(\bbZ^k,\bbZ^n)$ be such that $\rank_p(\bB) = k$ for every prime divisor $p$ of $\exp(G)$. 
    Then, there exists a constant $C >0$ depending only on $G$ such that for all $n\geq k$ 
    \begin{align}
        \lvert \#\{F\in \Sur_{\bbZ}(\bbZ^n,G): F\bB = 0\} - (\#G)^{n - k} \rvert \leq C \Bigl(\frac{\#G}{2}\Bigr)^{n}.  \nonumber
    \end{align}  
\end{lemma}
\begin{proof}
    We first argue that we can replace $\Sur_{\bbZ}(\bbZ^n,G)$ by $\Hom_{\bbZ}(\bbZ^n,G)$ up to a negligible error.  
    If a morphism $F:\bbZ^n\to G$ is not surjective then there exists some proper subgroup $H\subsetneq G$ such that $F(\bbZ^n) = H$. 
    Hence, since $\#\Hom_{\bbZ}(\bbZ^n,H) = (\# H)^n$,   
    \begin{align}
        \#(\Hom_{\bbZ}(\bbZ^n,G) \setminus \Sur_{\bbZ}(\bbZ^n,G)) &\leq \sum_{H \subsetneq G} \# \Hom_{\bbZ}(\bbZ^n,H) = \sum_{H \subsetneq G} (\# H)^{n}. 
    \end{align}   
    Denote $S_G$ for the number of proper subgroups of $G$. 
    Then, since $\#H \leq \# G/2$ by $H$ being a proper subgroup,
    \begin{align}
        \lvert \#\{F\in \Hom_{\bbZ}(\bbZ^n,G)&: F\bB = 0\} - \#\{F\in \Sur_{\bbZ}(\bbZ^n,G): F\bB = 0\} \rvert \label{eq: HomToSur} \\ 
        & \leq  \#(\Hom_{\bbZ}(\bbZ^n,G) \setminus \Sur_{\bbZ}(\bbZ^n,G))  \leq S_G \Bigl(\frac{\#G}{2}\Bigr)^{n}. \nonumber 
    \end{align}
    We next argue that $\#\{F\in \Hom_{\bbZ}(\bbZ^n,G): F\bB = 0\} = (\#G)^{n-k}$. 
    (This is not immediate because the columns $b_1,\ldots,b_k$ of $\bB$ are not necessarily part of a $\bbZ$-module basis for $\bbZ^n$.)

    Let $\bB = \bU \bD \bV$ be the Smith normal form of $\bB$. 
    This means that $\bU \in \bbZ^{n\times n}$ and $\bV \in \bbZ^{k\times k}$ are matrices with $\det(\bU), \det(\bV)\in \{-1,1\}$ and $\bD$ is an $n\times k$ diagonal matrix with integer diagonal entries satisfying $d_1 \mid d_2 \mid \ldots \mid d_k$.  
    For brevity, denote $a \de \exp(G)$.
    For every $p\mid a$, the assumption that $\rank_p(\bB) = k$ implies that $d_i \not\equiv 0 \bmod p$.
    It follows that the $d_i$ are multiplicative units for $\bbZ/a\bbZ$.
    Hence, the matrix $\bD' \de \diag(d_1,\ldots,d_k)$ is invertible in $(\bbZ/a\bbZ)^{k\times k}$.  
    Further, note that $\bU$ and $\bV$ are invertible over $\bbZ$.  
    Hence, if $u_1,\ldots,u_n$ are the columns of $\bU$ then the reductions to $(\bbZ/a\bbZ)^n$ determine a $(\bbZ/a\bbZ)$-module basis, and the reduction of $\bD'\bV$ to $(\bbZ/a\bbZ)^{k\times k}$ is invertible.
    Consequently, since $\bB =  (u_1,\ldots,u_k)\bD'\bV$, the reductions of $b_1,\ldots,b_k$ together with the reductions of the $u_i$ with $i\geq k+1$ determine a $(\bbZ/a\bbZ)$-module basis for $(\bbZ/a\bbZ)^n$. 

    Denote $\pi:\bbZ^n \to (\bbZ/a\bbZ)^n$ for the reduction map. 
    Then, since $a = \exp(G)$, it holds for every $F\in\Hom_{\bbZ}(\bbZ^n,G)$ that there is some unique $\overline{F}\in \Hom_{\bbZ/a\bbZ}((\bbZ/a\bbZ)^n, G)$ with $F = \overline{F} \circ \pi$.   
    Recall that for any ring $R$ an $R$-module morphism from a free $R$-module to an arbitrary $R$-module may be specified uniquely by arbitrarily specifying the images of the basis elements. 
    Consequently, since the $\pi(b_i)$ are part of a $(\bbZ/a\bbZ)$-module basis, 
    \begin{align}
        \#\{F\in \Hom_{\bbZ}(\bbZ^n,G): F\bB = 0\}  &= \#\{\overline{F}\in \Hom_{\bbZ/a\bbZ}((\bbZ/a\bbZ)^n,G): \overline{F}\circ \pi \circ \bB = 0\} \nonumber \\ 
        & = (\# G)^{n-k}.  \label{eq: HomFB0}
    \end{align}
    Combine \eqref{eq: HomToSur} with \eqref{eq: HomFB0} and set $C \de S_G$ to conclude the proof. 
\end{proof} 
We next estimate $\bbP(F\bX = xF)$. 
The quality of the estimates will be better when $F$ is ``very surjective''.
To make this precise, we rely on a notion of \emph{codes} which is due to Wood \cite{wood2017distribution} and a notion of \emph{robust morphisms} which is due to Nguyen and Wood \cite{nguyen2022random}. 

\subsubsection{Estimate for codes}\label{sec: CodesEstimate}
Let $e_1,\ldots,e_n \in \bbZ^n$ be the standard basis vectors. 
For any $\sigma \subseteq \{1,\ldots,n \}$ write $V_{\sigma}\de \spn_{\bbZ}(e_i : i \in \sigma)$ for the $\bbZ$-submodule of $\bbZ^n$ consisting of vectors whose nonzero coordinates are in $\sigma$. 
We abbreviate $V_{\setminus \sigma} \de V_{\{1,\ldots,n \}\setminus \sigma}$.   
\begin{definition}\label{def: Code}
    Let $G$ be an Abelian group. 
    Then, $F\in \Hom_{\bbZ}(\bbZ^n,G)$ is called a \emph{code of distance $w\geq 1$} if for every $\sigma \subseteq \{1,\ldots,n \}$ with $\# \sigma <w$ one has $F(V_{\setminus \sigma}) = G$.   
\end{definition}
The foregoing definition may also be applied to $\bbZ[x]$-modules since these can be viewed as Abelian groups through the $\bbZ$-module structure.    
\begin{lemma}\label{lem: CodeEstimate}
    Adopt assumptions \ref{a: X} and \ref{a: alpha_n} and fix a $\bbZ[x]$-module $N$ such that all prime divisors of $\#N$ are in $\sP$.
    Then, for every $\delta >0$ there exist constants $C,c>0$ such that for all $n\geq 1$ 
    \begin{align}
        \sum_{\substack{F\in \Hom_{\bbZ}(\bbZ^n,N)\\F\text{ a code of distance }\delta n}} \bigl\lvert \bbP(F\bX = xF) - (\#N)^{-n}\bigr\rvert \leq  C n^{-c}. \label{eq: CodeEstimate} 
    \end{align} 
\end{lemma}
\begin{proof}
    For any code $F$ and any $A\in \Hom_{\bbZ}(\bbZ^n,N)$ it is shown in \cite[Lemma 4.7]{nguyen2022random} that $\lvert \bbP(F\bX = A) -(\# N)^{-n}
    \rvert \leq C (\# N)^{-n} n^{-c}$.
    Actually, strictly speaking, \cite[Lemma 4.7]{nguyen2022random} is stated for matrices which are $\alpha_n$-balanced at all primes, but only balancedness at primes dividing $\#N$ is necessary for its proof. 
    (Indeed, \cite[Lemma 4.7]{nguyen2022random} follows from \cite[Lemma 4.5]{nguyen2022random} whose proof may be found in \cite[Lemma 2.1]{wood2019random} and only requires the weaker condition; see \cite[Definition 1]{wood2019random}.) 
    The result now follows since there are at most $\#\Hom_{\bbZ}(\bbZ^n,N) = (\#N)^n$ summands in \eqref{eq: CodeEstimate}.   
\end{proof}

\subsubsection{Estimate for non-codes}\label{sec: NonCodeEstimate}
The contribution of terms in \eqref{eq: SurMomentToProbSum} corresponding to non-codes turns out to be negligible.
It is however delicate to make this rigorous. 
The estimate which can be achieved on $\bbP(F\bX = xF)$ for a generic non-code $F$ is namely insufficient to beat the combinatorial factor corresponding to the number of non-codes. 
Hence, a subdivision of the non-codes is required to balance the quality of the estimates against the combinatorial costs. 

For an integer $d$ with prime factorization $d = \prod_{i} p_i^{e_i}$ denote $\ell(d) \de \sum_{i} e_i$.
Given a subgroup $H\subseteq G$, let $[G:H]\de \#G/\#H$ denote the index of $H$ in $G$.
\begin{definition}\label{def: Robust}
    Let $G$ be a finite Abelian group and let $\delta >0$ be a scalar. 
    Then, $F\in \Hom_{\bbZ}(\bbZ^n,G)$ is called \emph{$\delta$-robust for a subgroup $H\subseteq G$} if $H$ is minimal with the property that
    \begin{align}
        \#\bigl\{i \in \{1,\ldots,n \}: F(e_i) \notin H \bigr\} \leq \ell([G:H]) \delta n\label{eq: Def_Robust}
    \end{align} 
    That is, $H$ satisfies \eqref{eq: Def_Robust} and no strict subgroup $H' \subsetneq H$ satisfies \eqref{eq: Def_Robust}.    
\end{definition} 
The main motivation for \Cref{def: Robust} is the following property: if $F$ is $\delta$-robust for $H$ then the restriction of $F$ to $V_{\sigma}$ with $\sigma \de \{i: F(e_i) \in H \}$ is a code of distance $\delta n$ when $H$ is viewed as the codomain of this restriction.
Indeed, suppose this were not the case. 
Then, there exists $\mu \subseteq \sigma$ with $\# \mu < \delta n$ such that $H' \de F(V_{\sigma\setminus \mu})$ is a strict subgroup of $H$.
So, since $[G:H'] \geq [G:H] +1$,  
\begin{align}
    \#\{i\in \{1,\ldots,n \}: F(e_i) \not\in H' \} &
    \leq \#\{i\in \{1,\ldots,n \}:F(e_i)\not\in H \} + \#\mu\label{eq: RobustImpliesCode}\\
    \leq \ell([G:H'])\delta n\nonumber 
\end{align}
contradicting the minimality of $H$. 

In particular, any $F\in \Hom_{\bbZ}(\bbZ^n,G)$ which is not a code of distance $\delta n$ is not $\delta$-robust for $G$. 
However, \eqref{eq: Def_Robust} is always satisfied when $G = H$. 
This implies that any non-code has to be $\delta$-robust for some, not necessarily unique, proper subgroup of $G$. 
Hence, 
\begin{align}
    \{F\in \Sur_{\bbZ}(\bbZ^n,G):F &\text{ not a code of distance }\delta n \}\label{eq: NotACodeRobust} \\ 
    & \subseteq \bigcup_{H \subsetneq G}  \{F\in \Sur_{\bbZ}(\bbZ^n,G):F \text{ is }\delta\text{-robust for }H\}.\nonumber 
\end{align}
We next establish an estimate on $\bbP(F\bX = xF)$ when $F$ is $\delta$-robust for some $H$. 
The following lemma generalizes \cite[Lemma 4.11]{nguyen2022random} which concerns a similar bound for $\bbP(F(Y) = 0)$.
\begin{lemma}\label{lem: RobustBound}
    Fix scalars $\delta,\alpha >0$, an integer $n \geq 1$, and a finite Abelian group $G$. 
    Fix a proper subgroup $H\subsetneq G$, denote $d \de [G:H]$, and consider a maximal chain of proper subgroups 
    \begin{align}
        H = G_{\ell(d)} \subsetneq \cdots \subsetneq G_2 \subsetneq G_1 \subsetneq G_0 = G.\label{eq: HChain}
    \end{align}   
    Consider a $\delta$-robust morphism $F\in \Hom_{\bbZ}(\bbZ^n,G)$ for $H$.
    For every $1 \leq j \leq \ell(d)$ denote $p_j \de [G_{j-1} : G_{j}]$ and 
    \begin{align}
        w_j \de \#\bigl\{i\in \{1,\ldots,n \}: F(e_i) \in G_{j-1}\setminus G_j \bigr\}. \label{eq: Def_wj}
    \end{align}
    Abbreviate $a \de \exp(G)$ for the exponent of $G$.
    Then, for every $g\in G$ and every $\bbZ^n$-valued random vector $Y$ whose entries are independent and $\alpha$-balanced modulo all prime divisors of $a$,
    \begin{align}
        \bbP\bigl(F(Y) = g \bigr) \leq  &\Bigl( (\#G)^{-1}d + \exp(-\alpha \delta n/a^2) \Bigr) \label{eq: RobustBound}  \prod_{j=1}^{\ell(d)}\Bigl(p_j^{-1} + \frac{p_j -1}{p_j}\exp(-\alpha w_j/p_j^2) \Bigr).
    \end{align}  
\end{lemma}
\begin{proof}
    The strategy in this proof is to reduce to the case of codes where estimates are available from \cite[Lemma 4.5]{nguyen2022random}.  
    Again, similar to the remarks in the proof of \Cref{lem: CodeEstimate}, that lemma is stated for the case where $Y$ is balanced at all primes, but the case where $Y$ is merely balanced at prime divisors of $a$ follows from its proof.  

    For every $j\in \{1,2,\ldots,\ell(d)\}$ define a set of indices by 
    \begin{align}
        \sigma_j\de \bigl\{i\in \{1,\ldots,n \}: F(e_i) \in G_{j-1}\setminus G_j\bigr\}. \label{eq: Def_sigmaj}
    \end{align}
    Then, for every $r \leq \ell(d)$ the set of indices $i$ with $F(e_i)\notin G_r$ is given by $\Sigma_r \de \cup_{j=1}^r \sigma_j$.   
    Write $Y = (y_1,\ldots,y_n)$ and observe that $\sum_{i\notin \Sigma_r} y_i F(e_i) \in G_r$ for any $r \leq \ell(d)$. 
    Consequently, since $F(Y) = \sum_{i=1}^n y_i F(e_i)$, it is only possible to have $F(Y) = g$ if $\sum_{i\in \Sigma_r}y_i F(e_i) - g \in G_r$ for all $r \leq \ell(d)$.
    Hence, by definition of conditional probability,
    \begin{align}
        \bbP\Bigl(F(Y) = g\Bigr)&= \bbP\Bigl(\sum_{i\in \Sigma_1} y_i F(y_i) - g \in G_1\Bigr)\bbP\Bigl(F(Y) = g\,\Big\vert\, \sum_{i\in \Sigma_1} y_i F(y_i) - g \in G_1 \Bigr)\nonumber\\
        &= \prod_{j=1}^{\ell(d)} \bbP\Bigl(\sum_{i\in \Sigma_j} y_i F(y_i) - g \in G_j \,\Big\vert\, \forall r < j:\sum_{i\in \Sigma_r} y_i F(y_i) - g \in G_r\Bigr) \label{eq: Decomp}  \\ 
        &\quad \times \bbP\Bigl(F(Y) = g\,\Big\vert\, \forall r \leq \ell(d):\sum_{i\in \Sigma_r} y_i F(y_i) - g \in G_r \Bigr).  \nonumber 
    \end{align}
    We next bound the probabilities occurring in \eqref{eq: Decomp}.
    Recall that the $y_i$ are independent. 
    Hence, if we fix some $j \leq \ell(d)$ and condition on the values achieved by the $y_i$ with $i\in \Sigma_{j-1}$, then    
    \begin{align}
        \bbP&\Bigl(\sum_{i\in \Sigma_{j}}y_i F(e_i) - g \in G_{j}\, \Big\vert\, \forall r < j:\sum_{i\in \Sigma_r} y_i F(y_i) - g \in G_r\Bigr)\label{eq: TotalProbStep}\\ 
        &= \bbE\Bigl[\bbP\Bigl(\sum_{i\in \Sigma_{j}}y_i F(e_i) - g \in G_{j}\, \Big\vert\, y_i : i\in \Sigma_{j-1}\Bigr) \, \Big\vert\, \forall r < j:\sum_{i\in \Sigma_r} y_i F(y_i) - g \in G_r\Bigr]\nonumber\\ 
        &\leq  \max_{h\in G_{j-1}} \bbP\Bigl(\sum_{i\in \sigma_{j}}y_i F(e_i) - h \in G_{j}\Bigr).  \nonumber 
    \end{align}
    Here, the final step used that $\Sigma_{j}\setminus \Sigma_{j-1} =\sigma_j$ and $\sum_{i\in \Sigma_{j-1}}y_i F(e_i) -g$ was identified with $h$.  
    Denote $F_{j}:V_{\sigma_{j}} \to G_{j-1}/G_j$ for the map found by restricting $F$ to $V_{\sigma_j}$ and reducing modulo $G_{j}$. 
    Recall \eqref{eq: Def_wj} and note that $w_j = \#\sigma_j$. 
    We claim that $F_j$ is a code of distance $w_j$; recall \Cref{def: Code}. 
    Indeed, the maximality of \eqref{eq: HChain} ensures that $G_{j-1}/G_j$ is a cyclic group of prime order and consequently, for every $i\in \sigma_j$, $F_j(e_i)$ generates $G_{j-1}/G_j$ since $F_j(e_i) \not\equiv 0 \bmod G_j$ by definition of $\sigma_j$; recall \eqref{eq: Def_sigmaj}. 
    Now apply \cite[Lemma 4.5]{nguyen2022random} to $F_j$ and use that $p_j =\#(G_{j-1}/G_j)$ to find 
    \begin{align}
        \bbP\Bigl(\sum_{i\in \sigma_{j}}y_i F(e_i) - h \in G_{j}\Bigr) \leq \frac{1}{p_j} +  \frac{p_j - 1}{p_j} \exp\Bigl(-\frac{\alpha w_j}{p_j^2}\Bigr).
    \end{align}
    Using this bound on the product in \eqref{eq: Decomp} yields the product in \eqref{eq: RobustBound}. 
    It remains to bound the remaining factor.
    Here, similarly to \eqref{eq: TotalProbStep}, we have 
    \begin{align}
        \bbP\Bigl(F(Y) = g\,\Big\vert\, \forall r \leq \ell(d):\sum_{i\in \Sigma_r} y_i F(y_i) - g \in G_r \Bigr) \leq \max_{h \in G_{\ell(d)}}\bbP\Bigl(\sum_{i\not\in \Sigma_{\ell(d)}}y_i F(e_i) = h  \Bigr).\nonumber    
    \end{align} 
    By the argument preceding \eqref{eq: RobustImpliesCode}, the restriction of $F$ to $V_{\setminus \Sigma_{\ell(d)}}$ defines a code of distance $\delta n$.
    Hence, by \cite[Lemma 4.5]{nguyen2022random} and the fact that $\exp(H)\mid \exp(G)$, 
    \begin{align}
        \bbP\Bigl(\sum_{i\not\in \Sigma_{\ell(d)}}y_i F(e_i) = h  \Bigr) \leq (\#H)^{-1} + \exp(-\alpha \delta n/a^2). 
    \end{align}
    It was here used that $(\#H -1)/\#H \leq 1$. 
    Use that $\# H = \#G/d$ to conclude the proof.   
\end{proof}
\begin{corollary}\label{cor: RobustMatrixBound}
    Adopt assumption \ref{a: X} and let $N$ be a $\bbZ[x]$-module such that all prime divisors of $\#N$ are in $\sP$.   
    Then, for every subgroup $H\subseteq N$ and every $F\in \Hom_{\bbZ}(\bbZ^n,N)$ which is $\delta$-robust for $H$, 
    \begin{align}
        \bbP(F\bX = xF) \leq& \bigl( (\#N)^{-1}d + \exp\bigl(-\alpha_n \delta n/a^2\bigr) \bigr)^n   \prod_{j=1}^{\ell(d)}\Bigl(p_j^{-1} + \frac{p_j -1}{p_j}\exp(-\alpha_n w_j/p_j^2) \Bigr)^n\nonumber
    \end{align}
    where $d,p_j,w_j$, and $a$ are defined as in \Cref{lem: RobustBound} with $G = N$.  
\end{corollary}
\begin{proof} 
    Since the entries of $\bX$ are assumed to be independent, one has that $\bbP(F\bX = xF) = \prod_{i=1}^n \bbP(F(\bX e_i) = xF(e_i))$.
    The result is hence immediate from \Cref{lem: RobustBound} applied with $Y \de \bX e_i$ and $g \de xF(e_i)$. 
\end{proof}
\begin{lemma}\label{lem: NotCodeBound}
    Let $N$ be a $\bbZ[x]$-module such that all prime divisors of $\#N$ are in $\sP$ and adopt assumptions \ref{a: X} and \ref{a: alpha_n}.   
    Then, there exists $\delta_0 >0$ such that for every $\delta < \delta_0$ there exist constants $C,c >0$ such that for all $n\geq 1$ 
    \begin{align}
        \sum_{\substack{F\in \Sur_{\bbZ}(\bbZ^n,N) \\  F\text{ not a code of distance }\delta n}} \negquad 
        \bbP(F\bX = xF) \leq C n^{-c}. \nonumber
    \end{align}  
\end{lemma} 
\begin{proof}
    By \eqref{eq: NotACodeRobust}, one may upper bound the sum as 
    \begin{align}
        \sum_{\substack{F\in \Sur_{\bbZ}(\bbZ^n,N) \\  F\text{ not a code of distance }\delta n}} \negquad
        \bbP(F\bX = xF) \leq \sum_{H\subsetneq N}\sum_{\substack{F\in \Sur_{\bbZ}(\bbZ^n,N) \\  F\text{ is }\delta\text{-robust for }H}}\bbP(F\bX = xF).  \label{eq: StepSplitNotCodeRobust} 
    \end{align} 
    Fix some proper subgroup $H\subsetneq N$ and pick a maximal chain of subgroups $G_j$ as in \eqref{eq: HChain} with $G=N$.    
    We denote $d \de [N:H]$. 

    By \cite[Lemma 4.10]{nguyen2022random}, the number of $F\in \Hom_{\bbZ}(\bbZ^n,N)$ which are $\delta$-robust for $H$ and satisfy that there are exactly $w_j$ indices $i\leq n$ with $F(e_i) \in G_{j-1}\setminus G_j$ for $1 \leq j \leq \ell(d)$ is at most $(\# H)^{n-\sum_{j=1}^{\ell(d)}w_j} \prod_{j=1}^{\ell(d)} \binom{n}{w_j}(\# G_{j-1})^{w_j}$.
    When $F$ is surjective we have $w_1 \neq 0$. 
    Further, when $F$ is $\delta$-robust for $H$ we have $w_j \leq \ell(d) \delta n$ for all $j\leq \ell(d)$. 
    Indeed, if this were not the case then we would have $\#\{i\leq n: F(e_i) \not\in H \} > \ell(d)\delta n$ which contradicts \eqref{eq: Def_Robust}.   
    Now, by the combination of \Cref{cor: RobustMatrixBound} with the aforementioned count on the number of $\delta$-robust morphisms, 
    \begin{align}
        \sum_{\substack{F\in \Sur_{\bbZ}(\bbZ^n,N) \\  F\text{ is }\delta\text{-robust for }H}}&\bbP(F\bX = xF)\leq \negsp\sum_{\substack{0 \leq w_1,\ldots,w_{\ell(d)}\leq \ell(d) \delta n\\ w_1 \neq 0} }\negsp (\#H)^{n-\sum_{j=1}^{\ell(d)}w_j} \prod_{j=1}^{\ell(d)} \binom{n}{w_j}(\# G_{j-1})^{w_j}\label{eq: SumFRobustH}  \\ 
        &\times  \bigl( (\#N)^{-1}d + \exp\bigl(-\alpha_n \delta n/a^2\bigr) \bigr)^n \prod_{j=1}^{\ell(d)}\Bigl(p_j^{-1} + \frac{p_j -1}{p_j}\exp(-\alpha_n w_j/p_j^2) \Bigr)^n.\nonumber 
    \end{align} 
    It remains a nontrivial task to compute the right-hand side of \eqref{eq: SumFRobustH}. 
    Fortunately, a related sum was considered by Nguyen and Wood \cite{nguyen2022random} and we can extract the relevant estimate from their proofs.
    More precisely, the sum in \eqref{eq: SumFRobustH} is a special case of the sum which occurs in the first centered equation of the proof of \cite[Theorem 4.12]{nguyen2022random}: take $u=0$ in their notation. 
    Following the arguments word-for-word up to the centered inequality at the end of page 23 in \cite{nguyen2022random} now yields the desired result. 
\end{proof} 
\begin{lemma}\label{lem: NotCodeSecondBound}
    Adopt assumptions \ref{a: X} and \ref{a: alpha_n} and fix a $\bbZ[x]$-module $N$ such that all prime divisors of $\#N$ are in $\sP$.
    Then, there exists $\delta_0 >0$ such that for every $\delta < \delta_0$ there exist constants $C,c >0$ such that for all $n\geq 1$ 
    \begin{align}
        \sum_{\substack{F\in \Sur_{\bbZ}(\bbZ^n,N) \\  F\text{ not a code of distance }\delta n}}\negquad \lvert 
        \bbP(F\bX = xF) - (\#N)^{-n} \rvert \leq C n^{-c}.\nonumber 
    \end{align}  
\end{lemma}
\begin{proof}
    By \Cref{def: Code}, if $F$ is not a code of distance $\delta n$ then we can find some $\sigma \subseteq \{1,\ldots,n \}$ and a proper subgroup $H \subsetneq N$ such that $F(V_{\setminus \sigma}) \subseteq H$ and $\# \sigma = \lfloor \delta n \rfloor$.  
    Hence, since $F$ is uniquely determined by the images of the basis elements of $\bbZ^n$, 
    \begin{align}
        \sum_{\substack{F\in \Sur_{\bbZ}(\bbZ^n,N) \\  F\text{ not a code of distance }\delta n}}\negquad \negquad (\#N)^{-n}  \leq \sum_{H \subsetneq N} \binom{n}{\lfloor \delta n\rfloor}  (\# H)^{n - \lfloor \delta n \rfloor} (\# N)^{-n + \lfloor \delta n \rfloor}.\label{eq: NotCodeSur}      
    \end{align}
    The binary entropy bound \cite[Eq.(7.14), p.151]{cover1999elements} implies that $\binom{n}{\lfloor \delta n\rfloor} \leq 2^{n E(\lfloor \delta n \rfloor/n)}$ where $E(y)\de -y\log_2(y) + (1-y)\log_2(1-y)$.
    Note that $\lim_{y \to 0} E(y) = 0$. 
    We can hence find some sufficiently small $\delta_0$ such that \Cref{lem: NotCodeBound} is applicable and $\delta_0 + E(\lfloor \delta_0 n \rfloor/n) < 1/2$ for all $n\geq 1$. 
    Then, since $\#H \leq \#N /2$, we have for every $\delta < \delta_0$ that 
    \begin{align}
        \sum_{\substack{F\in \Sur_{\bbZ}(\bbZ^n,N) \\  F\text{ not a code of distance }\delta n}}\negquad \negquad (\#N)^{-n}  \leq S_N 2^{(E(\lfloor \delta n\rfloor /n ) + \delta - 1)n} \leq S_N 2^{-n/2}\label{eq: AppliedBinaryEntropy}  
    \end{align}  
    where $S_N$ is the number of proper subgroups of $N$.
    Let $C',c'>0$ be such that $S_N 2^{-n/2} \leq C' n^{-c'}$ for all $n\geq 1$ and use \Cref{lem: NotCodeBound} together with the triangle inequality to conclude the proof.   
\end{proof}
\subsubsection{Combining the estimates}\label{sec: CombineEstimates}
We finally combine all preceding estimates to complete the proof of \Cref{prop: LimitingMoments}. 
\begin{proof}[Proof of \texorpdfstring{\Cref{prop: LimitingMoments}}{Proposition}]
    By \eqref{eq: SurMomentToProbSum} and the triangle inequality,     
    \begin{align}
        \lvert \bbE[\#\Sur_{\bbZ[x]}(\coker(\tilde{\bW}),N)] - (\#{}&{} N)^{-k} \rvert \leq \negsp\sum_{F\in  \Sur_{\bbZ}(\bbZ^n,N): F\bB = 0}\left\lvert (\#N)^{-n} - \bbP(F\bX = xF)  \right\rvert \label{eq:TediousVolt}\\
        &\quad +\left\lvert  (\# N)^{-n}\#\{F\in \Sur_{\bbZ}(\bbZ^n,N): F\bB = 0\}- (\# N)^{-k}   \right\rvert.  \nonumber 
    \end{align}
    Pick some $\delta >0$ which is sufficiently small to ensure that \Cref{lem: NotCodeSecondBound} is applicable.
    Then, by the triangle inequality,   
    \begin{align}
        &\sum_{F\in  \Sur_{\bbZ}(\bbZ^n,N): F\bB = 0}\lvert (\#N)^{-n} - \bbP(F\bX = xF)  \rvert  \label{eq: SurFB0CodeNotCode}\\ 
        &\quad \leq\negsp \sum_{\substack{F\in \Hom_{\bbZ}(\bbZ^n,N) \\ F\text{ a code of distance }\delta n}}\negquad \lvert (\#N)^{-n} - \bbP(F\bX = xF)  \rvert 
         +\negsp \sum_{\substack{F\in \Sur_{\bbZ}(\bbZ^n,N) \\ F\text{ not a code of distance }\delta n}} \negquad \negquad \lvert (\#N)^{-n} - \bbP(F\bX = xF)\rvert .\nonumber
    \end{align}
    Let $c,C>$ be the constants from \Cref{lem: CodeEstimate} and let $c',C'>0$ be the constants from \Cref{lem: NotCodeSecondBound}. 
    Then, the right-hand side of \eqref{eq: SurFB0CodeNotCode} is at most $Cn^{-c} + C'n^{-c'}$ and hence tends to zero as $n$ tends to infinity. 
    Further, by \Cref{lem: NumSubFB0}, there exists a constant $C''>0$ such that
    \begin{align}
        \left\lvert  (\# N)^{-n}\#\{F\in \Sur_{\bbZ}(\bbZ^n,N): F\bB = 0\}- (\# N)^{-k}   \right\rvert \leq C'' 2^{-n}.\label{eq: SurFB}
    \end{align}
    Remark that the right-hand side of \eqref{eq: SurFB} tends to zero as $n$ tends to infinity to conclude the proof. 
\end{proof}
\subsection{Solving the moment problem}\label{sec: SolveMoment}
We next apply a general result concerning measures on categories of \cite[Theorem 1.6]{sawin2022moment} to invert the moment problem. 
Using \cite[Lemma 6.1]{sawin2022moment} and \cite[Lemma 6.3]{sawin2022moment}, that result may be specialized to our context---namely to limiting measures on the category of finite modules with $N$-moments given by $(\#N)^{-k}$. 
Let us state this specialization explicitly for the sake of definiteness:
\begin{lemma}[Special case of {\cite[Theorem 1.6]{sawin2022moment}}]\label{lem: SawinWoodSpecialCase}
    Let $R$ be a ring and consider a sequence of random finite $R$-modules $X_{n}$ such that for every fixed finite $R$-module $N$ 
    \begin{align*}
        \lim_{n\to \infty} \bbE[\#\Sur_{R}(X_n, N)] = (\#N)^{-k}.
    \end{align*} 
    Let $S$ be quotient ring of $R$ with $\#S < \infty$ and let $L_1,\ldots,L_r$ be representatives of the isomorphism classes of finite simple $S$-modules. 
    Further, denote $q_i$ for the cardinality of the endomorphism field of $L_i$ for every $i\leq r$.
    Then, for every finite $S$-module $N$,
    \begin{align}
        \lim_{n\to \infty}\bbP(X_n\otimes_{R}S \cong N)
        &= \frac{1}{ (\# N)^k\# \Aut_S(N)} \prod_{i=1}^r \prod_{j=1}^\infty \Bigl(1 - \frac{q_i^{-j}\#\Ext^1_S(N,L_i) }{\#\Hom_{S}(N,L_i) (\# L_i)^k} \Bigr).\nonumber
    \end{align}
\end{lemma}
Let us further remark that it follows from the statement of \cite[Theorem 1.6]{sawin2022moment} that the limiting measure in \Cref{lem: SawinWoodSpecialCase} has $N$-moments given by $(\#N)^{-k}$. 
In particular, the $N$-moment for $N = \{0 \}$ is equal to one which implies that the limiting measure is a probability measure.   

It now remains to combine \Cref{prop: LimitingMoments} with \Cref{lem: SawinWoodSpecialCase} and to simplify the results. 
This can be accomplished with a direct computation. 
Note that the following result implies \Cref{thm: Sparse_MainResult} as the special case with $k=1$.   
\begin{proposition}\label{prop: InvertMomentSawinWood}
    Adopt assumptions \ref{a: X} to \ref{a: alpha_n}. 
    Fix a positive integer $m\geq 1$, a monic polynomial $Q\in \bbZ[x]$ of degree $\geq 1$, and introduce $Q_{i,p}\in \bbF_p[x]$, $r_p\geq 1$, and $d_{i,p} = \deg Q_{i,p}$ using the factorization of $Q$ modulo $p$ as in \Cref{thm: Sparse_MainResult}.      

    For every $p\in \sP$ denote $S_{p^m} \de \bbZ[x]/( p^m\bbZ[x] + Q(x)\bbZ[x] )$. 
    Then, given a finite $S_{p^m}$-module $N_{p^m, Q}$ for every $p\in \sP$, it holds that  
    \begin{align*}
        \lim_{n\to \infty}\bbP(\forall p\in \sP:&\coker(\tilde{\bW})_{p^m,Q} \cong  N_{p^m,Q}) = \prod_{p\in \sP}\biggl( \frac{1}{(\#N_{p^m,Q})^k\#\Aut_{S_{p^m}}(N_{p^m,Q})}\\ 
        &\quad \times \prod_{i = 1}^{r_p} \prod_{j=1}^\infty \Bigl(1-\frac{\#\Ext_{S_{p^m}}^1\bigl(N_{p^m,Q}, \bbF_p[x]/(Q_{i,p}(x)\bbF_p[x]) \bigr)}{\#\Hom_{S_{p^m}}\bigl(N_{p^m,Q}, \bbF_p[x]/(Q_{i,p}(x)\bbF_p[x]) \bigr)} p^{-(k+j)d_{i,p} } \Bigr)\biggr).
    \end{align*} 
\end{proposition}
\begin{proof}
    Denote $R \de \bbZ[x]/((\prod_{p\in \sP} p^m)\bbZ[x])$ and note that for any $\bbZ[x]$-module $M$ and any $R$-module $N$ one has a bijection between $\Sur_{R}(M\otimes_{\bbZ[x]}R,N)$ and $\Sur_{\bbZ[x]}(M ,N)$.
    Consequently, by \Cref{prop: LimitingMoments}, we have that for every finite $R$-module $N$   
    \begin{align}
        \lim_{n\to \infty}\bbE[\#\Sur_R(\coker(\tilde{\bW})\otimes_{\bbZ[x]}R,N)] = (\# N)^{-k}. \label{eq:RoyalLeaf}
    \end{align}
    Let $S\de R/(Q(t)R)$ and define a finite $S$-module by $N \de \oplus_{p\in \sP} N_{p^m,Q}$. 
    The Chinese remainder theorem yields $R \cong \oplus_{p\in \sP} \bbZ[x]/(p^m \bbZ[x])$ which implies that $S = \oplus_{p\in \sP}S_{p^m}$ and $\coker(\tilde{\bW})\otimes_{\bbZ[x]}S \cong \oplus_{p\in \sP} \coker(\tilde{\bW})_{p^m, Q}$.
    Consequently, by \Cref{lem: SawinWoodSpecialCase} and the fact that isomorphism occurs if and only if the corresponding summands are isomorphic, 
    \begin{align}
        \lim_{n\to \infty}\bbP\bigl(\forall p\in \sP:&\coker(\tilde{\bW})_{p^m,Q} \cong  N_{p^m,Q}\bigr) = \lim_{n\to \infty}\bbP\bigl((\coker(\tilde{\bW})\otimes_{\bbZ[x]} R) \otimes_{R} S \cong N\bigr)\nonumber    \\ 
        &\  = \frac{1}{(\# N)^k\# \Aut_S(N) } \prod_{i=1}^r \prod_{j=1}^\infty \Bigl(1 - \frac{\#\Ext^1_S(N,L_i) }{\#\Hom_{S}(N,L_i) (\# L_i)^k}q_i^{-j} \Bigr). \label{eq:NewRug}
    \end{align}
    Simple modules over $S$ are precisely the modules of the form $S/m$ with $m$ a maximal ideal of $S$. 
    Further, maximal ideals of $\bbZ[x]$ are of the form $m = p\bbZ[x] + f(x)\bbZ[x]$ with $p$ a prime and $f(x)\in \bbZ[x]$ irreducible modulo $p$ of degree $\geq 1$ \cite[p.22]{reid1995undergraduate}. 
    Hence, since maximal ideals of $S$ are in one-to-one correspondence with maximal ideals of $\bbZ[x]$ which contain $\prod_{p\in \sP}p^m$ and $Q$, the modules $L_i$ in \eqref{eq:NewRug} are of the form $\bbF_p[x]/(Q_{i,p}(x)\bbF_p[x])$. 
    Consequently, by \eqref{eq:NewRug} and the fact that the endomorphism field of $\bbF_p[x]/(Q_{i,p}(x)\bbF_p[x]))$ is isomorphic to $\bbF_p[x]/(Q_{i,p}(x)\bbF_p[x])$ which is a finite field of order $p^{d_{i,p}}$,    
    \begin{align}
        &\lim_{n\to \infty}\bbP\bigl(\forall p\in \sP:\coker(\tilde{\bW})_{p^m,Q} \cong  N_{p^m,Q}\bigr) \label{eq:SafeHat}\\ 
        &= \frac{1}{(\#N)^k\#\Aut_{S}(N)}\prod_{p\in \sP}  \prod_{i = 1}^{r_p} \prod_{j=1}^\infty \Bigl(1-\frac{\#\Ext_{S}^1\bigl(N, \bbF_p[x]/(Q_{i,p}(x)\bbF_p[x]) \bigr)}{\#\Hom_{S}\bigl(N, \bbF_p[x]/(Q_{i,p}(x)\bbF_p[x]) \bigr)} p^{-(k+j)d_{i,p} } \Bigr).  \nonumber
    \end{align}
    Here, observe that $\#N = \prod_{p\in \sP} \#N_{p^m,Q}$, $\#\Aut_S(N)=\prod_{p\in \sP}\#\Aut_{S_{p^m}}(N_{p^m,Q})$, and note that $\#\Hom_{S}\bigl(N, \bbF_p[x]/(Q_{i,p}(x)\bbF_p[x]) \bigr) = \#\Hom_{S_{p^m}}(N_{p^m,Q}, \bbF_p[x]/(Q_{i,p}(x)\bbF_p[x]))$. 
    Further, since $\Ext$ takes with direct sums in the first argument to products \cite[Proposition 3.3.4]{weibel1994introduction}, 
    \begin{align} 
        \#\Ext_S^1(N,\bbF_p[x]/(Q_{i,p}(x)\bbF_p[x])) \label{eq:SlowCat} 
        &= \prod_{q\in \sP} \#\Ext_{S}^1(N_{q^m,Q},\bbF_p[x]/(Q_{i,p}(x)\bbF_p[x])). 
    \end{align} 
    Finally, using the definition of $\Ext^1$ in terms of short exact sequences \cite[Theorem 3.4.3]{weibel1994introduction} together with the fact that finite $S$-modules correspond to tuples of $S_{p^m}$-modules, one can verify that 
    \begin{align}
        \#\Ext_{S}^1(N_{q^m,Q},\bbF_p[x]/(Q_{i,p}(x)\bbF_p[x])) &= 
        \begin{cases}    
            \#\Ext_{S_{p^m}}^1(N_{p^m,Q},\bbF_p[x]/(Q_{i,p}(x)\bbF_p[x])) & \text{if }p=q,\\ 
            1 & \text{else}.
        \end{cases}\label{eq:NormalMom}
    \end{align}
    Combine \eqref{eq:SafeHat}--\eqref{eq:NormalMom} to conclude the proof.     
\end{proof}

\section{Future work}\label{sec: FutureWork}
Ultimately, we would like to show that the conditions of \Cref{thm: Wang} are satisfied with nonvanishing probability.
The current paper makes progress in this direction: we now have concrete proof techniques which can be used to study cokernels of walk matrices.
There are however still a number of interesting open problems. 

For instance, it remains entirely open to understand whether the condition $\coker(\bW)_{2^2} \cong (\bbZ/2\bbZ)^{\lfloor n/2\rfloor}$ is often satisfied, even heuristically.
This is because the distribution of $\coker(\bW)_{2^m}$ is highly sensitive to the graph being simple which makes approximation by results for directed graphs inadequate.
Indeed, when $\bX$ is the adjacency matrix of a simple graph and $b=e$ is the all-ones vector, then \cite[Lemma 14]{wang2013generalized} implies that $\rank_2\bW \leq \lceil n/2\rceil$. 
Equivalently, it holds that $\coker(\bW)_2 \cong(\bbZ/2\bbZ)^{\ell}$ for some $\ell \geq \lfloor n/2\rfloor$. 
This is very different behavior from the distribution for directed graphs in \Cref{thm: MainResult_Uniform} with $p=2$ since the latter is concentrated on small groups.

For odd primes, numerical evidence suggest that the difference is not as severe. 
\Cref{tab: ErdosRenyiExperiment} namely suggest that the distribution of $\coker(\bW)_{p^m}$ has qualitatively similar behavior for simple and directed graphs. 
Quantitatively, however, a close inspection shows that there is a small but nonzero difference which does not seem to disappear when $n$ grows large, suggesting that the limiting distribution for simple graphs differs from the one for random directed and weighted graphs. 
For instance, the estimated probabilities that $\coker(\bW)_{p^2}\in \{0,\bbZ/p\bbZ \}$ for $p=3$ is $0.758\pm 0.002$ whereas \Cref{thm: MainResult_Uniform} predicts $0.747$.

So, the extension of our results to the setting of simple graphs poses an interesting problem, both for odd and even primes. 
We intend to pursue this in future work.
Let us finally recall \Cref{conj: OddPrimeWellBehaved} and note that a proof of this conjecture would also be valuable since the underlying challenges are also likely to show up in the study of simple graphs. 
In support of this conjecture, we present numerical data in \Cref{tab: DirectedExperiment} which suggests that the conclusion of \Cref{thm: MainResult_Uniform} remains valid for unweighted directed graphs.

\section*{Acknowledgments}
I would like to thank Nils van de Berg, Jaron Sanders, and Haodong Zhu for providing helpful feedback on a draft of this manuscript. 
I further thank Aida Abiad for an inspiring talk which motivated my interest in the spectral determinacy of graphs.

This work is part of the project Clustering and Spectral Concentration in Markov Chains with project number OCENW.KLEIN.324 of the research programme
Open Competition Domain Science – M which is partly financed by the Dutch Research
Council (NWO). 

\newpage 

\begin{table}[t]
    \begin{center}
    \resizebox{0.85\textwidth}{!}{
    \begin{tabular}{ |c|c|c|c|c|c|c|c| }
        \hline
        $p$&$n =10$&$n=12$&$n=15$&$n = 20$&$n=30$&$n=40$&\Cref{thm: MainResult_Uniform}\\
        \hline
        3&$0.495$&$0.625$&$0.726$&$0.757$&$0.756$&$0.758$&$0.746834\ldots$\\
        5& $0.549$&$0.725$&$0.869$&$0.913$&$0.914$&$0.915$&$0.912399\ldots$\\
        7& $0.563$&$0.750$&$0.906$&$0.953$&$0.956$&$0.957$&$0.956337\ldots$\\ 
        11&$0.571$&$0.765$&$0.930$&$0.981$&$0.983$&$0.983$&$0.982726\ldots$\\  
        \hline
      \end{tabular}}
    \end{center}
    \caption{Probability that $\coker(\bW)_{p^2}\in \{0,\bbZ/p\bbZ \}$ when $\bX$ is the adjacency matrix of an undirected Erd\H{o}s--R\'enyi random graph on $n$ nodes and $b = e$, estimated based on $10^5$ independent samples. 
    The estimated values have an uncertainty of $\pm 0.002$.
    Also displayed is the limiting probability in the case of directed and weighted random graphs which follows from \Cref{thm: MainResult_Uniform} with $m=2$.
    Computation of the group structure of $\coker(\bW)_{p^2}$ was done using the algorithm \texttt{smith\_form} in \texttt{SageMath} \cite{sagemath}.    
    }
    \label{tab: ErdosRenyiExperiment}
    \end{table}
    \begin{table}[t]
        \begin{center}        
        \resizebox{0.85\textwidth}{!}{
        \begin{tabular}{ |c|c|c|c|c|c|c|c| }
            \hline
            $p$&$n =10$&$n=12$&$n=15$&$n = 20$&$n=30$&$n=40$&\Cref{thm: MainResult_Uniform}\\
            \hline
            3&$0.650$&$0.707$&$0.737$&$0.746$&$0.749$&$0.747$&$0.746834\ldots$\\
            5&$0.759$&$0.844$&$0.898$&$0.911$&$0.911$&$0.912$&$0.912399\ldots$\\
            7&$0.786$&$0.881$&$0.940$&$0.956$&$0.957$&$0.956$&$0.956337\ldots$\\ 
            11&$0.802$&$0.901$&$0.965$&$0.982$&$0.983$&$0.983$&$0.982726\ldots$\\  
            \hline
          \end{tabular}}
        \end{center}
        \caption{Probability that $\coker(\bW)_{p^2}\in \{0,\bbZ/p\bbZ \}$ when $\bX\sim \Unif\{0,1 \}^{n\times n}$ is the adjacency matrix of an unweighted directed random graph and $b = e$.
        The same comments as in the caption of \Cref{tab: ErdosRenyiExperiment} apply: the estimation used $10^5$ independent samples, there is an uncertainty of $\pm 0.002$, and \texttt{SageMath} \cite{sagemath} was used. 
        }
        \label{tab: DirectedExperiment}
        \end{table}

\bibliographystyle{plainurl}

\begin{thebibliography}{11}

    \bibitem{abiad2012cospectral}
    {A.~Abiad} and {W.H.~Haemers}.
    Cospectral graphs and regular orthogonal matrices of level 2.
    {\em {T}he {E}lectronic {J}ournal of {C}ombinatorics}, 2012.
    \href {https://doi.org/10.37236/2383} {\path{doi:10.37236/2383}}.
    
    \bibitem{butler2010note}
    {S.~Butler}.
     A note about cospectral graphs for the adjacency and normalized
      {L}aplacian matrices.
     {\em {L}inear and {M}ultilinear {A}lgebra}, 2010.
     \href {https://doi.org/10.1080/03081080902722741}
      {\path{doi:10.1080/03081080902722741}}.
    
    \bibitem{cheong2023cokernel}
    {G.~Cheong} and {M.~Yu}.
     The distribution of the cokernel of a polynomial evaluated at a
      random integral matrix.
     {\em arXiv preprint arXiv:2303.09125v3}, 2023.
     \href {https://doi.org/10.48550/arXiv.2303.09125}
      {\path{doi:10.48550/arXiv.2303.09125}}.
    
    \bibitem{clancy2015cohen}
    {J.~Clancy}, {N.~Kaplan}, {T.~Leake}, {S.~Payne}, and {M.~M. Wood}.
     {O}n a {C}ohen--{L}enstra heuristic for {J}acobians of random graphs.
     {\em {J}ournal of {A}lgebraic {C}ombinatorics}, 2015.
     \href {https://doi.org/10.1007/s10801-015-0598-x}
      {\path{doi:10.1007/s10801-015-0598-x}}.
    
    \bibitem{cohen2006heuristics}
    {H.~Cohen} and {H.W.~Lenstra}.
     Heuristics on class groups of number fields.
     In {\em Number Theory Noordwijkerhout 1983}. Springer, 1984.
     \href {https://doi.org/10.1007/BFb0099440}
      {\path{doi:10.1007/BFb0099440}}.
    
    \bibitem{cover1999elements}
    {T.M.~Cover} and {J.A.~Thomas}.
     {\em {E}lements of {I}nformation {T}heory}.
     John Wiley \& Sons, second edition, 1999.
     \href {https://doi.org/10.1002/047174882X}
      {\path{doi:10.1002/047174882X}}.
    
    \bibitem{godsil2012controllable}
    {C.D.~Godsil}.
     Controllable subsets in graphs.
     {\em {A}nnals of {C}ombinatorics}, 2012.
     \href {https://doi.org/10.1007/s00026-012-0156-3}
      {\path{doi:10.1007/s00026-012-0156-3}}.
    
    \bibitem{godsil1982constructing}
    {C.D.~Godsil} and {B.D.~McKay}.
     Constructing cospectral graphs.
     {\em {A}equationes {M}athematicae}, 1982.
     \href {https://doi.org/10.1007/BF02189621}
      {\path{doi:10.1007/BF02189621}}.
    
    \bibitem{haemers2004enumeration}
    {W.H.~Haemers} and {E.~Spence}.
     Enumeration of cospectral graphs.
     {\em European Journal of Combinatorics}, 2004.
     \href {https://doi.org/10.1016/S0195-6698(03)00100-8}
      {\path{doi:10.1016/S0195-6698(03)00100-8}}.
    
    \bibitem{halbeisen1999generation}
    {L.~Halbeisen} and {N.~Hungerb{\"u}hler}.
     Generation of isospectral graphs.
     {\em {J}ournal of {G}raph {T}heory}, 1999.
     \href {https://doi.org/fdb8jq} {\path{doi:fdb8jq}}.
    
    \bibitem{koval2023exponentially}
    {I.~Koval} and {M.~Kwan}.
     Exponentially many graphs are determined by their spectrum.
     {\em {T}he {Q}uarterly {J}ournal of {M}athematics}, 2023.
     \href {https://doi.org/10.1093/qmath/haae030}
      {\path{doi:10.1093/qmath/haae030}}.
    
    \bibitem{li2021arithmetic}
    {S.~Li} and {W.~Sun}.
     An arithmetic criterion for graphs being determined by their
      generalized {$A_{\alpha}$}-spectra.
     {\em {D}iscrete {M}athematics}, 2021.
     \href {https://doi.org/10.1016/j.disc.2021.112469}
      {\path{doi:10.1016/j.disc.2021.112469}}.
    
    \bibitem{liu2022unlocking}
    {F.~Liu} and {J.~Siemons}.
     Unlocking the walk matrix of a graph.
     {\em {J}ournal of {A}lgebraic {C}ombinatorics}, 2022.
     \href {https://doi.org/10.1007/s10801-021-01065-3}
      {\path{doi:10.1007/s10801-021-01065-3}}.
    
    \bibitem{nguyen2020surjectivity}
    {H.H.~Nguyen} and {E.~Paquette}.
     Surjectivity of near-square random matrices.
     {\em Combinatorics, Probability and Computing}, 2020.
     \href {https://doi.org/10.1017/S0963548319000348}
      {\path{doi:10.1017/S0963548319000348}}.
    
    \bibitem{nguyen2022random}
    {H.H.~Nguyen} and {M.M.~Wood}.
     Random integral matrices: universality of surjectivity and the
      cokernel.
     {\em {I}nventiones mathematicae}, 2022.
     \href {https://doi.org/10.1007/s00222-021-01082-w}
      {\path{doi:10.1007/s00222-021-01082-w}}.
    
    \bibitem{o2016conjecture}
    {S.~O'Rourke} and {B.~Touri}.
     On a conjecture of {G}odsil concerning controllable random graphs.
     {\em {S}{I}{A}{M} {J}ournal on {C}ontrol and {O}ptimization}, 2016.
     \href {https://doi.org/10.1137/15M1049622}
      {\path{doi:10.1137/15M1049622}}.
    
    \bibitem{qiu2021oriented}
    {L.~Qiu}, {W.~Wang}, and {W.~Wang}.
     Oriented graphs determined by their generalized skew spectrum.
     {\em {L}inear {A}lgebra and its {A}pplications}, 2021.
     \href {https://doi.org/10.1016/j.laa.2021.03.033}
      {\path{doi:10.1016/j.laa.2021.03.033}}.
    
    \bibitem{qiu2023smith}
    {L.~Qiu}, {W.~Wang}, and {H.~Zhang}.
     Smith normal form and the generalized spectral characterization of
      graphs.
     {\em {D}iscrete {M}athematics}, 2023.
     \href {https://doi.org/10.1016/j.disc.2022.113177}
      {\path{doi:10.1016/j.disc.2022.113177}}.
    
    \bibitem{reid1995undergraduate}
    {M.~Reid}.
     {\em Undergraduate commutative algebra}.
     Cambridge University Press, 1995.
     \href {https://doi.org/10.1017/CBO9781139172721}
      {\path{doi:10.1017/CBO9781139172721}}.
    
    \bibitem{sawin2022moment}
    {W.~Sawin} and {M.M.~Wood}.
     The moment problem for random objects in a category.
     {\em arXiv preprint arXiv:2210.06279v2}, 2022.
     \href {https://doi.org/10.48550/arXiv.2210.06279}
      {\path{doi:10.48550/arXiv.2210.06279}}.
    
    \bibitem{schwenk1973almost}
    {A.J.~Schwenk}.
     Almost all trees are cospectral.
     {\em New directions in the theory of graphs}, 1973.
    
    \bibitem{sundaram2012structural}
    {S.~Sundaram} and {C.N.~Hadjicostis}.
     Structural controllability and observability of linear systems over
      finite fields with applications to multi-agent systems.
     {\em {I}{E}{E}{E} {T}ransactions on {A}utomatic {C}ontrol}, 2012.
     \href {https://doi.org/10.1109/TAC.2012.2204155}
      {\path{doi:10.1109/TAC.2012.2204155}}.
    
    \bibitem{sagemath}
    {The Sage Developers}.
     {\em {S}ageMath, the {S}age {M}athematics {S}oftware {S}ystem
      ({V}ersion 9.3)}, 2021.
     URL: \url{https://www.sagemath.org}.
    
    \bibitem{van2003graphs}
    {E.R.~van Dam} and {W.H.~Haemers}.
     {W}hich graphs are determined by their spectrum?
     {\em {L}inear {A}lgebra and its {A}pplications}, 2003.
     \href {https://doi.org/10.1016/S0024-3795(03)00483-X}
      {\path{doi:10.1016/S0024-3795(03)00483-X}}.
    
    \bibitem{verbitsky2023canonization}
    {O.~Verbitsky} and {M.~Zhukovskii}.
     Canonization of a random circulant graph by counting walks.
     {\em {I}nternational {C}onference and {W}orkshops on {A}lgorithms and {C}omputation}, 2024.
     \href {https://doi.org/10.1007/978-981-97-0566-5_23}
      {\path{doi:10.1007/978-981-97-0566-5_23}}.
    
    \bibitem{wang2013generalized}
    {W.~Wang}.
     {G}eneralized {S}pectral {C}haracterization of {G}raphs {R}evisited.
     {\em {T}he {E}lectronic {J}ournal of {C}ombinatorics}, 2013.
     \href {https://doi.org/10.37236/3748} {\path{doi:10.37236/3748}}.
    
    \bibitem{wang2017simple}
    {W.~Wang}.
     A simple arithmetic criterion for graphs being determined by their
      generalized spectra.
     {\em {J}ournal of {C}ombinatorial {T}heory, {S}eries {B}}, 2017.
     \href {https://doi.org/10.1016/j.jctb.2016.07.004}
      {\path{doi:10.1016/j.jctb.2016.07.004}}.
    
    \bibitem{wang2022haemers}
    {W.~Wang} and {W.~Wang}.
     Haemers' conjecture: an algorithmic perspective.
     {\em {E}xperimental {M}athematics}, 2024.
     \href {https://doi.org/10.1080/10586458.2024.2337229}
      {\path{doi:10.1080/10586458.2024.2337229}}.
    
    \bibitem{wang2006excluding}
    {W.~Wang} and {C.-X.~Xu}.
     An excluding algorithm for testing whether a family of graphs are
      determined by their generalized spectra.
     {\em Linear algebra and its applications}, 2006.
     \href {https://doi.org/10.1016/j.laa.2006.01.016}
      {\path{doi:10.1016/j.laa.2006.01.016}}.
    
    \bibitem{wang2006sufficient}
    {W.~Wang} and {C.-X.~Xu}.
     A sufficient condition for a family of graphs being determined by
      their generalized spectra.
     {\em European Journal of Combinatorics}, 2006.
     \href {https://doi.org/10.1016/j.ejc.2005.05.004}
      {\path{doi:10.1016/j.ejc.2005.05.004}}.
    
    \bibitem{weibel1994introduction}
    {C.A.~Weibel}.
     {\em An introduction to homological algebra}.
     Cambridge University Press, 1994.
     \href {https://doi.org/10.1017/CBO9781139644136}
      {\path{doi:10.1017/CBO9781139644136}}.
    
    \bibitem{wood2017distribution}
    {M.M.~Wood}.
     The distribution of sandpile groups of random graphs.
     {\em {J}ournal of the {A}merican {M}athematical {S}ociety}, 2017.
     \href {https://doi.org/10.1090/jams/866}
      {\path{doi:10.1090/jams/866}}.
    
    \bibitem{wood2019random}
    {M.M.~Wood}.
     Random integral matrices and the {C}ohen--{L}enstra heuristics.
     {\em {A}merican {J}ournal of {M}athematics}, 2019.
     \href {https://doi.org/10.1353/ajm.2019.0008}
      {\path{doi:10.1353/ajm.2019.0008}}.
    
    \end{thebibliography}

\end{document}